\documentclass[letterpaper,11pt,oneside,reqno]{amsart}
\usepackage{graphicx, yhmath}
\usepackage{amssymb}
\usepackage{amsthm, color}
\usepackage{bm}
\usepackage{amsmath, enumerate}
\usepackage{amsfonts}
\theoremstyle{plain}
    \newtheorem{theorem}{Theorem}
    
    \newtheorem{lemma}[theorem]{Lemma}

\theoremstyle{definition} 
    
    \newtheorem{remark}[theorem]{Remark}

\def\beqs{\begin{eqnarray*}}
\def\eeqs{\end{eqnarray*}}

\def \tK {\tilde K}
\def \Z {\mathbb Z}

\def \R {\mathbb R}

\newcommand {\ind}{1\!\! 1}
\newcommand {\indd}[1]{1\!\! 1_{\{#1\}}}

\newcommand{\la}{\langle}
\newcommand{\ra}{\rangle}
\renewcommand{\ll}{\la\!\la}
\newcommand{\rr}{\ra\!\ra}

\newcommand{\rrr}{\kappa}

\newcommand{\und}[1]{\underline{#1}}

\newcommand{\be}{\begin{equation}}
\newcommand{\ee}{\end{equation}}

\newcommand{\SSEP}{S_0}

\renewcommand{\l}{\langle}
\renewcommand{\r}{\rangle }

\newcommand{\heta}{\hat \eta}

\newcommand{\cM}{\mathcal M}

\title{Diffusivity of lattice gases}
\author{Jeremy Quastel}
\address{\hskip-13pt Departments of Mathematics and Statistics, University of Toronto
\newline e-mail:  \rm \texttt{quastel@math.toronto.edu}}

\author{Benedek Valk\' o}
\address{\hskip-13pt Department of Mathematics, University of Wisconsin -- Madison
\newline e-mail:  \rm \texttt{valko@math.wisc.edu }}

\begin{document}

\begin{abstract}
We consider one component lattice gases with a local dynamics and a stationary product Bernoulli measure. We give upper and lower bounds on the diffusivity at an equilibrium point depending on the dimension and the local behavior of the macroscopic flux function. We show that if the model is expected to be diffusive, it is indeed diffusive, and, if it is expected to be superdiffusive, it is indeed superdiffusive.
\end{abstract}

\maketitle


\section{Introduction}

In this article we consider lattice gas, or particle models for the stochastically forced diffusion equation with non-linear drift in $\R^d$
\begin{equation}\label{sbemodeleq}
\partial_t \rho = \nabla \cdot j(\rho) + \nabla\cdot D\nabla \rho + \sqrt{D}\nabla\cdot \xi
\end{equation}
where $j(\rho)$ is  the macroscopic current, $D$ is a diffusivity  and $\xi_i$, $i=1,\dots, d$ are smooth processes
approximating independent space-time white noises. 
 Our main interest is on the long time, large wavelength behaviour of solutions, in particular
 of the space-time correlation functions in equilibrium, and how this depends on the non-linearity $j(\rho)$ and the dimension $d$.

  {\em In one dimension only}, (\ref{sbemodeleq}) coincides with
the  stochastic Hamilton-Jacobi equation for a height function $h$ defined through $\nabla h=\rho$,
\begin{equation}\label{shjbe}
\partial_t h =  j(\nabla h) + D \Delta h + \sqrt{D}\xi.
\end{equation}
The case $j(x) = \nu |x|^2$ with white noise forcing is the {\it Kardar-Parisi-Zhang (KPZ) equation}.  
In one dimension, it has a deep connection with representation theory and  integrable systems, which has recently led to some distributions being computed exactly \cite{corwin}. Note however that the  scaling of the  stochastic Hamilton-Jacobi equations in higher 
dimensions are different from (\ref{sbemodeleq}).  Most importantly, with white noise forcing, 
(\ref{sbemodeleq}) have formal white
noise invariant measures in all dimensions, while (\ref{shjbe}) have them only in $d=1$.

To explain the heuristics for the scalings,  we write an approximate equation for the  rescaled variable 
$$
\rho_\epsilon(t,x) = \epsilon^{-\alpha} \rho( \epsilon^{-\beta} t, \epsilon^{-1} x)
$$
with a small parameter $\epsilon$ representing the ratio of small to large scales, and $\alpha$, $\beta\ge 0$, \begin{equation}\label{approx}
\partial_t \rho_\epsilon = \epsilon^{\alpha-\beta+1}\sum_{i=1}^d \epsilon^{-\alpha}  j_i'(\epsilon^\alpha\rho_\epsilon) \partial_{x_i} \rho_\epsilon + \epsilon^{-\beta+2}\nabla\cdot D\nabla \rho_\epsilon + \epsilon^{\frac{-\beta+2(1-\alpha) +d}{2}}\sqrt{D}\nabla\cdot \xi.
\end{equation}
The approximation is in the last term which we scaled as if it had
 correlations $E[\xi(t,x)\xi(s,y)]=\delta(t-s)\delta(x-y)$.  In this case of Gaussian white noise,  the rescaled distribution
 $\xi( a t, b x)$ would be statistically the same as $a^{-1/2} b^{-d/2} \xi(t,x)$.

Now one has several natural choices for $\alpha$ and $\beta$ depending on the dimension $d$, the non-linearity, as well as what it is that one desires to 
see.

\underline{Hyperbolic (Euler) scaling ($\beta=1$).}  If we take $\alpha =0$ we obtain 
\begin{equation}
\partial_t \rho_\epsilon = \nabla\cdot   j(\rho_\epsilon)  + \epsilon\nabla\cdot D\nabla \rho_\epsilon + \epsilon^{\frac{d+1}{2}}\sqrt{D}\nabla\cdot \xi,
\end{equation}
from which it is not hard to guess  that $\rho_\epsilon\to \rho$ satisfying \begin{equation}
\label{foql}
\partial_t \rho= \nabla\cdot   j(\rho) .
\end{equation}
Weak solutions of (\ref{foql}) are not unique.  However they are, for bounded $\rho_0$, \cite{MR0216338,MR0267257} if supplemented
by the entropy condition, for all $c\in \R$, in the  weak sense,
\begin{equation}\label{entropy}
\partial_t |\rho -c|\le \nabla \cdot {\rm sgn}( \rho - c) ( j(\rho) - j(c)).
\end{equation}
It is expected that the Euler scaling of asymmetric interacting particle models
leads to the entropy solutions in wide generality, but it has only been verified in special cases
\cite{MR1130693, MR1681094, MR2285732, MR2744889, MR2578381}

\underline{Longer time scales ($\beta>1$).} If we look on longer time scales, the scaling now depends on the 
behaviour of $j(\rho)$ near $\rho_0=0$, as well as on the dimension $d$.   We will 
 expand $j$ around $\rho_0=0$, and note that in the resulting Taylor series, the first two terms $j(0)$ and $ j'(0)\rho$ can be easily removed from the equation by a simple change of variables and a coordinate shift.  So we set $j(0)= j'(0)=0$.   The first
non-trivial term in the series is therefore $\tfrac12j''(0)\rho^2$.
Generically, one has $j''(0)\neq 0$, in which case  $j'(\epsilon^\alpha\rho_\epsilon)\sim \epsilon^\alpha\rho_\epsilon j''(0)$.  This means that we can take $\beta =2$ (the diffusive scale) and $\alpha =1$ to get, 
\begin{equation}
\partial_t \rho_\epsilon =\sum_{i=1}^d \epsilon^{-\alpha}  j_i'(\epsilon^\alpha\rho_\epsilon) \partial_{x_i} \rho_\epsilon +\nabla\cdot D\nabla \rho_\epsilon + \epsilon^{\frac{d-2}{2}}\sqrt{D}\nabla\cdot \xi.
\end{equation}
In $d=1$, the noise dominates, and clearly there is no limit.  $d=2$ is critical, and one expects
logarithmic divergences.  Only in $d\ge 3$ is there a limit \cite{MR1301374,MR1465163},
\begin{equation}
\partial_t \rho =\nabla \cdot j''(0) \rho +\nabla\cdot D\nabla \rho .
\end{equation}
So we can obtain a limit by diffusive scaling only if $d\ge3$.

To see the fluctuations, one is  guided by the natural rescaling of the invariant white noise, which
leads one to guess that $\alpha$ should be $d/2$.  We arrive at
\begin{equation}\label{du}
\partial_t \rho_\epsilon = \epsilon^{\frac{d+2}2 -\beta} \sum_{i=1}^d \epsilon^{-d/2}j_i'(\epsilon^{d/2}\rho_\epsilon) \partial_{x_i} \rho_\epsilon + \epsilon^{2-\beta}\nabla\cdot D\nabla \rho_\epsilon + \epsilon^{\frac{2-\beta}{2}}\sqrt{D}\nabla\cdot \xi.
\end{equation}
In $d\ge 3$, the clear choice is $\beta=2$, and the fluctuations are given by  an (infinite dimensional) Ornstein-Uhlenbeck process \cite{MR2070098,MR2073338}
\begin{equation}
\partial_t \rho = \nabla\cdot D\nabla \rho +\sqrt{D}\nabla\cdot \xi.
\end{equation}
In $d=1$, we are constrained by the first term.  In the generic case $j''(0)\neq 0$, we are forced to
take $\beta = 3/2$.  This is the KPZ scaling.  Note however that  the limit process is  {\it not} the naive guess, which would be the entropy solution of
\begin{equation}
\partial_t \rho = \tfrac12j''(0)\partial_x \rho^2
\end{equation}
because it can be checked by the Lax-Oleinik formula that this does not preserve (standard) white noise \cite{FM_burgers}, 
while (\ref{du}) does, for each $\epsilon>0$.  In fact, the limit process is a Burgers equation 
forced by  a highly non-trivial, and not well understood residual noise.    Some finite dimensional distributions are known, but not all  \cite{Corwin:2011ly}.  

One can see that $d=2$ is critical for (\ref{du}), and as usual in the critical case one expects logarithmic corrections.
On the other hand, if we are at an inflection point, then the behaviour is different.  If $j_i''(0)=0$, 
then $\epsilon^{-\alpha}  j_i'(\epsilon^{\alpha}\rho_\epsilon)$ is of order $\epsilon^{ \alpha}$
if $j_i'''(0)\neq 0$ and of order  $\epsilon^{ 2\alpha}$ if $j_i'''(0)= 0$.  One then checks that the above
$d\ge 3$ scalings  now apply in $d=2$ if all $j_i''(0)=0$ and in $d=1$ if $j''(0)=j'''(0)=0$.
These are the diffusive cases.
The final case is that of a generic inflection point in $d=1$.  We have $j''(0)=0$ but $j'''(0)\neq 0$.
Referring to (\ref{du}), we have $ \epsilon^{-d/2}  j'(\epsilon^{d/2}\rho_\epsilon)\sim \epsilon^{d/2}$
and we can see once again that $\beta=2$ but we are in a critical case and therefore there 
should be logarithmic corrections.

The model cases of (\ref{sbemodeleq}) are thus $j(\rho) =v\cdot \rho^n$ 
with $n=2$ the generic stochastic Burgers equation, diffusive in $d\ge 3$, logarithmically super-diffusive in $d=2$ and KPZ super-diffusive in $d=1$;  $n=3$ is the generic inflection case, diffusive
in $d\ge 2$, and logarithmically super-diffusive in $d=1$; and $n\ge 4$ is the double inflection case,
diffusive in all dimensions. 

But these arguments are at best heuristic, because (\ref{approx}) is at best approximate, relying on the exact rescaling of pure white noise 
forcing $\xi$.    However, no existence or uniqueness theorem in known for (\ref{sbemodeleq}) with white noise forcing for non-trivial $j(\rho)$ except for the very special  $d=1$, $j(\rho) =\rho^2$ case, and even
here one is restricted to a finite interval with periodic boundary conditions \cite{Hairer:2011bs}.  In fact, the expectation is that (\ref{sbemodeleq})  does {\it not} have non-trivial
solutions except in that special case.  
There are several ways around this.  One could try to deal with the  smoothed  noise, however such a study is made difficult by
the fact that we
do not know the resulting invariant measures, which are used extensively in the analysis. 
Lattice gas models provide natural discretizations
 where one does know the invariant measures, so they are a natural place to start.  
As we will see, the superdiffusive behaviour is reflected in a diffusion coefficient which we can study. Our method is based on the Green-Kubo formula, which requires
one to solve a resolvent equation in the space generated by the invariant measures,
or, equivalently, to study space-time correlations of the flux. This requires detailed knowledge of
the invariant measures. 

The superdiffusive scalings in the lattice gas models were identified at the physical level using the formalism of mode-mode coupling
theory.  Let $\hat\rho(k,t)$ be the spatial Fourier transform of the density $\rho(t,x)$.  $\hat{S}(k,t)= \langle \hat\rho(-k,0)\hat\rho(k,t)\rangle$ is called the 
intermediate scattering function.  The brackets denote expectation in equilibrium.  An equation 
can be written for $\partial_t \hat{S}(k,t)$, but, of course, it involves higher order correlations, equations
for the time evolution of which involve yet higher order correlations, etc.
The mode-mode coupling formalism approximates the higher order correlations at the first step by multilinear forms in $\hat{S}$, producing a closed equation.  A
description of the procedure can be found in \cite{berne}.  The result for (\ref{sbemodeleq})  with $j(\rho)=v\cdot\rho^n$  is 
\begin{equation}\label{mceq}
\partial_t \tilde{S}(k,t)= -D |k|^2 \tilde{S}(k,t) - c(v\cdot k)^2 (\tilde{S}(k,0))^{-1} \int_0^t   \tilde{S}(k,t-s) \underbrace{\tilde{S}\ast\cdots\ast \tilde{S}}_{n~{\rm times}} (k,s) ds
\end{equation}
where $\tilde{S}$ indicates that this is only an approximate equation for $\hat{S}$. The $\ast$
denotes convolution in $k$,  $c$ is a constant.  Assuming $\tilde{S}(k,t) = \tilde{S}(0,0) e^{
-D|k|^2 t - a (v\cdot k)^2 t(\log t)^\zeta}$ and then solving for $a$ and $\zeta$ by identifying both
sides of (\ref{mceq}) as $k\to \infty$ one guesses that  $\zeta=2/3$ in the critical case $n=2$, $d=2$ and $\zeta= 1/2$ in the critical case $n=3$, $d=1$.  These predictions were 
then backed up by extensive numerical simulation  \cite{vBKS85}.

Our results can be summarized the following way. For a large class of lattice gas models, which have 
Bernoulli invariant measures, if the model is expected to be diffusive, it is indeed diffusive, in a sense to 
be made precise in the next section.  And, if it is expected to be superdiffusive, it is indeed superdiffusive.

Recently, there has been a large amount of activity around such models in connection to 
the Kardar-Parisi-Zhang equation.  For example, 
\cite{GoncalvesJara} study the speed-changed exclusion models for the special $d=1$, $j(\rho) = c\rho^2$
case of
 (\ref{sbemodeleq}) or (\ref{shjbe}).  Following \cite{Bertini-Giacomin}, they study the weakly asymmetric
 limit, proving that under that scaling the density fields are tight, and satisfy  a plausible weak version of the
 Kardar-Parisi-Zhang equation, though, unfortunately, such solutions are not known to be unique.  In comparison, our work focuses on the case of general flux $j$ and dimension $d$, and  we do not introduce  small parameters into the problem, but study the scale of fluctuations directly.   In this 
 sense, our results are more in the spirit of, but far less exact, than the scaling results for the special
 solvable one dimesional models such as TASEP, PNG and q-TASEP.  On the other hand, and this has
 to be emphasized, we are in a situation where there appears to be no exact solvabilty, and, except in the special $d=1$, $j(\rho) = c\rho^2$
case, there is not even a known solvable model in the universality class.  One can of course ask if
some of these methods can be extended to models with more than one particle per site, such as 
zero-range, or where more than one particle moves at the same time.   
 We have not pursued this here, but as long as one has the explicit invariant measures,  the method
 should work, though it could be extremely complicated in practice.  One can also study related
 continuous spin models \cite{Herbert}.  For Gaussian invariant measures, they turn out to be very similar to the exclusion cases; the case of general potentials has not been studied yet. 
 
 A finer question than the diffusivity would be the limiting fluctuation fields.  As usual, what one expects is 
 Gaussian fluctuations in the diffusive and logarithmically super-diffusive cases.  The diffusive case
 is studied in \cite{MR2073338}.  The logarithmically super-diffusive case is still open. Only in the case of two
 dimensional exclusion with asymmetry in one direction are the scaling results fine enough that 
 one might hope to prove the Gaussian fluctuation limit.  In other cases, one still has to get the scale
 correct.

\section{Model and main results}

We now describe the model and results precisely.  We consider lattice gas or speed changed exclusion models on $\Z^d$ with local interaction and exclusion rule. 
They consist of particles performing continuous time random walks, with jump rate
depending on the local configuration, and with the exclusion rule that jumps to occupied sites
are suppressed.  
We think of it as a Markov process with state space is $\{0,1\}^{\Z^d}$, the $0$ or $1$ indicating the absence or presence of a particle at $x
\in \Z^d$. The infinitesimal generator is given by 
\begin{equation}\label{generator}
L f(\eta) =\sum_x \sum_{y}  r(y,\tau_{-x} \eta) \, \eta_x (1-\eta_{x+y})(f(\eta^{x,x+y})-f(\eta))
\end{equation}
where $\tau_x$ denotes the shift $(\tau_x\eta)_y = \eta_{y-x}$, and
$r:\Z^d\times \{0,1\}^{\Z^d}\to \R_+$ gives the  rate $r(y,\tau_{-x} \eta)$ of a particle
at $x$ to jump to $x+y$. $\eta_x (1-\eta_{x+y})$ indicates that there is a particle at $x$ and
no particle at $x+y$, so that the jump can be performed, and $f(\eta^{x,x+y})-f(\eta)$ measures
the change in the function $f$ from the pre-jump configuration $\eta$, to the post-jump configuration
$\eta^{x,x+y}$, which has the occupation variables at $x$ and $x+y$ exchanged.

We need to make some  assumptions on the jump rates:
%
%

{\it Local evolution.} For every $y\neq 0$ the function $r(y,\cdot)\ge 0$ is a local function, i.e. only depends on the values $\eta_z, |z|\le K$ where $K$ is a fixed constant.  Note that we can assume without loss of generality that $r(y,\eta)$ does not depend on  $(\eta_0, \eta_y)$.  Moreover,  $r(y,\cdot)=0$ if $|y|\ge K$.

{\it Divergence condition\footnote{Sometimes referred to as the {\it gradient condition} \cite{GoncalvesJara}.  However, it is not the same as the gradient condition used in the theory of hydrodynamic scaling limits.}.} 
There exist local functions $R_1,\dots, R_d$ so that 
\begin{equation}\label{grad_cond}
 \sum_{y} r(y, \eta) \, \left(\eta_{y} -\eta_0 \right) =\sum_{i=1}^d \nabla_{e_i} R_i(\eta)
\qquad {\rm where} \qquad 
\nabla_{e_i} f(\eta)=f(\eta)-f(\tau_{e_i} \eta).
\end{equation}

{\it Coercivity.} 
The additive group generated by those $y\in \Z^d$ for which $r(y, \eta)+r(-y, \tau_y\eta)>0$ for all $\eta$ is equal to $\Z^d$.

Let us comment on the conditions.  The locality of the rates is natural.  It just means that a particle looks in a finite neighbourhood to determine its jump rate.  The fact that under this condition there is a unique Markov process $\eta(t)$ on $\{0,1\}^{\Z^d}$ with
 $e^{tL} f(\eta) = E[ f(\eta(t))~|~\eta(0)=\eta]$,  follows from general theory \cite{liggett}.

 Assuming we have local rates, the divergence condition implies that  the  product Bernoulli($\rho$) measures, $\pi_\rho$, $\rho\in (0,1)$ are invariant.   This is our key assumption, which allows us to start the analysis.  It is not a generic property, but on the other hand one can find a rich family of examples.   The proof that it implies the invariance is as follows:  By general theory the local functions form a core for the generator \cite{liggett}, so
to check the invariance one need only prove that for such functions $\int Lf d\pi_\rho=0$. Using the fact that the change of variables $\eta\to \eta^{x,y}$ preserves $\pi_\rho$ so 
$
E_{\pi_\rho}  [r(y,\tau_{-x} \eta) \, \eta_x (1-\eta_{x+y})f(\eta^{x,x+y})]=E_{\pi_\rho} [r(y,\tau_{-x} \eta) \, \eta_{x+y} (1-\eta_{x})f(\eta)]
$
and therefore
\[
\int  Lfd\pi_\rho=E_{\pi_\rho} [f(\eta) \sum_{x,y}r(y,\tau_{-x} \eta) \, (\eta_{x+y} -\eta_{x})]=E_{\pi_\rho} [f(\eta)\sum_{i=1}^d \sum_{x} \nabla_{e_i} R_i(\tau_{-x} \eta)].
\]
The summation in $x$ gives a telescoping sum and by choosing a large enough box  the surviving terms from the sum lie well outside the support of $f$, and are therefore independent of it.  Hence the last term vanishes.  In fact, for local rates
the other direction is also true; the invariance of the Bernoulli measures implies the divergence condition. This will not be used in our results; the proof of the statement is left to the interested reader.   The divergence condition is the analogue in our context of the fact that the flow generated by a divergence free vector field preserves
Lebesgue measure.

 The coercivity condition is not a strong assumption.  It is equivalent to the ergodicity of the process and also equivalent to the symmetric part of the generator $L$  being comparable to the generator 
\begin{equation}
\SSEP f(\eta) =\sum_{x}\sum_{| y|=1} \, \eta_x (1-\eta_{x+y})(f(\eta^{x,x+y})-f(\eta))
\end{equation}
of the symmetric simple exclusion process.  More precisely, for the symmetric part $S=\frac{L+L^*}{2}$ 
there exist $0<c_1\le c_2<\infty$ such that the Dirichlet forms satisfy
\begin{equation}\label{S_comp}
c_1\la f(-\SSEP)f\ra\le \la f(-S)f\ra\le c_2 \la f(-\SSEP)f\ra.
\end{equation}
The Dirichlet forms are given by 
\begin{equation}
\la f(-\SSEP)f\ra=\tfrac12E_{\pi_\rho}[\sum_{x,e_i} (f(\eta^{x,x+e_i})-f(\eta))^2]
\end{equation}
and
$\la f(-S)f\ra = \tfrac12E_{\pi_\rho}  [\sum_{x,y}( r(y,\tau_{-x} \eta)+r(-y,\tau_{-x+y}\eta)(f(\eta^{x,x+y})-f(\eta))^2]$.
We can perform  $\eta^{x,x+y}$ by doing $||y||_1=m$ nearest neighbour switches.
By the Cauchy-Schwarz inequality, 
$
(f(\eta^{x,x+y})-f(\eta))^2\le m \sum_{j=0}^{m-1} ( f(\eta_j^{z_j,z_{j+1}})-f(\eta_j))^2
$. Since 
$r(\cdot, \cdot)$ is bounded, we get the right inequality of (\ref{S_comp}). The left inequality works 
similarly, but here we have to replace the switch $x\leftrightarrow x+e_i$ using jumps from the set $ \{y: r(y, \eta)+r(-y, \tau_y\eta)>0 \textup{ for all } \eta\}$. The 
coercivity condition makes this possible. \medskip

We now give some explicit examples of rates.
\begin{enumerate}
\item \emph{Finite range exclusion process.} Here $r(y,\eta)=p(y)$, this clearly satisfies the first two conditions. If we assume $p(\cdot)$ is such that the associated random walk can get to any site in
$\Z^d$ then the coercivity condition is satisfied as well.   The special case $p(1)=1_{|y|=1}$ is  the
symmetric simple exclusion process.

\item The following rates provide a simple model satisfying the conditions in one dimension without being a finite range exclusion process 
\begin{equation}
r(1,\eta)=3-\eta_{-1}-\eta_2, \qquad r(-1,\eta)=2, \qquad r(y,\eta)=0 \textup{\quad if\quad} |y|\neq1.\label{simplerates}
\end{equation}
At certain points of the proof we will illustrate the computations using this simple model before proceeding with the general proof.

\item Let $d=1$ and $0< y\in \Z$. Then it is easy to check that both $r(y,\eta)=\eta_1\cdots\eta_{y-1}$ and $(1-\eta_1)\cdots(1-\eta_{y-1})$ satisfy  the divergence condition. (Similar products can also be defined in the $y<0$ case.)
This gives the following model: suppose that we have an exclusion process satisfying the conditions, but we increase the jump rates of size $y$ with a certain constant $c_y$ if there are no particles (or no holes) between the starting point and the end point. If $\{y:c_y>0\}$ is finite  then this model will satisfy all the conditions\footnote{This is similar to, but not the same as the model q-TASEP, where $r(y,\eta)=0$ unless $y=1$, in which case
$r(1,\eta) = \sum_{n=0}^\infty (1-q^n) \eta_{n+2}\prod_{i=2}^{n+1} (1-\eta_i)$. The q-TASEP  does {\it not} 
satisfy the divergence condition (as well as having infinite range), and the invariant measures are
{\it not} Bernoulli  (see \cite{BorodinCorwin}).}.

\item It is easy to construct higher dimensional models from the one dimensional examples. Suppose we have one dimensional  rate functions $r_1, r_2, \dots, r_d$ which satisfy the conditions. Let $e_1, \dots, e_d$ denote the usual unit coordinate vectors in $\Z^d$ and let $P_i \eta$ be the projection of $\eta\in \{0,1\}^{\Z^d}$ to its $i^{th}$ coordinate. Now define the $d$-dimensional rate function $r$ the following way:
\begin{align}
r(y e_i,\eta)=r_i(y, P_i \eta), \quad y\in \Z \quad \textup{and} \quad r(z, \eta)=0 \quad\textup{ if } z\neq y e_i.
\end{align}
It is not hard to check that this will satisfy our conditions. 

\item The model (\ref{simplerates}) can also be extended into higher dimensions easily. Consider a $d$-dimensional exclusion with jump law $p_y, y\in \Z^d$ satisfying the conditions. Now modify the jump rates so that for a finite subset of $y$'s the jump rate is changed to 
\begin{align}
r(y,\eta)=p_y(3-\eta_{-y}-\eta_{2y})
\end{align}
The resulting model satisfies all our conditions. 
\end{enumerate}
 \medskip

We now define the  flux.  To define the microscopic flux we first write
\begin{equation}
L\eta_0= \sum_{y}  \left( 
r(y, \tau_{y}\eta) \eta_{-y}(1-\eta_0)
-
r(y, \eta) \eta_0(1-\eta_y)\right)\label{flux}
\end{equation}
Note that this is a finite sum and it can be written as a divergence $\sum_{i=1}^d \nabla_{e_i} W_i$ with local functions $W_i$. The vector $W=(W_1,\dots,W_d)$ is called the \emph{microscopic flux}\footnote{Note that although $W$ is not unique, we can just fix a function which is a linear combination of terms of the form $ r(y, \tau_{-z} \eta) \eta_{z+y}(1-\eta_{z})$.}. 
%
%
%
The flux for the adjoint $L^*$ is defined analogously; we denote it by $W^*=(W_1^*,\dots,W_d^*)$. 
The \emph{macroscopic flux} is
\begin{equation}
j(\rho)=E_\rho (W_1,\dots,W_d).
\end{equation}
The entries of this vector are polynomials in $\rho$, 
\begin{equation}
j_i(\rho)=\sum_y y\cdot e_i \,E_\rho r(y, \eta) \eta_0(1-\eta_{y}).
\end{equation}
The model is \emph{asymmetric} in the $i^{th}$ direction if $j_i(\rho)$ is not a constant.   Note the abuse of language in that there  are non-symmetric models with zero drift which are therefore not asymmetric.  So we should really say ``non-zero drift".  However, the term
asymmetric has become ubiquitous in the field and we will stick with it.
We emphasize that in order to identify the flux, one needs to know at least something about the invariant
measures.  In fact, we do not know any model where the flux has been identified without knowing the 
invariant measures exactly.

{\it Examples.} For a 1d exclusion process with jump rate $p(\cdot)$ the microscopic flux is the linear combinations of terms of the form  $\eta_x(1-\eta_{x+y})$ and the macroscopic flux is $b \rho(1-\rho)$ where $b$ is the first moment of $p(\cdot)$. For the simple 1d model defined in (\ref{simplerates}) the microscopic flux is $\eta_0(1-\eta_1)(1-\eta_{-1}-\eta_{-2})$ and the macroscopic flux is $j(\rho)=(1-2\rho)\rho(1-\rho)$. Note that this has an inflection point at $\rho=1/2$.\medskip

In general, one is 
interested in the large scale behaviour of the space-time correlations,
\begin{equation}
S(x,t)=\chi^{-1} \la \eta(x,t);\eta(0,0)\ra
\end{equation}
where $\chi(\rho)= \rho(1-\rho)$ and $ \la f;g\ra$ is just a notation for the equilibrium covariance 
$E[fg] - E[f]E[g]$. Note that we will also use the notation $ \la f,g\ra$ for the scalar product $E fg$.  The expectations are with respect to $\pi_\rho$ with some fixed $\rho\in(0,1)$.  
One easily checks \cite{MR1901953}
\begin{equation}
\sum_x S(x,t)=1,\quad \textup{and} \quad \sum_x x S(x,t)=j'(\rho) t.
\end{equation}
So the first non-trivial moment\footnote{Note however, that only in the attractive case (e.g. exclusion) will one necessarily have $S(x,t)\ge 0$. In that case $D(t) t$ is the variance of the position of the second class particle at time $t$, so it is actually connected to the second moment of a random variable. } is the time dependent diffusivity  
\begin{equation}\label{dif}
D(t)=t^{-1}{\chi}^{-1}\left(\sum_x x^2 \la \eta_x(t), \eta_0(0)-\rho \ra\right)-(j'(\rho) t)^2.
\end{equation}
Note the normalization so that standard diffusive behaviour corresponds to $D(t)=O(1)$ and
superdiffusive behaviour to $D(t)\to \infty$ as $t\to\infty$.
For technical reasons our key observable will be the Laplace transformed diffusivity 
 given by 
%
the Green-Kubo formula (see Lemma \ref{l:GK} in the Appendix), \begin{equation}\label{taub}
{\hat D}_{ii}(\lambda):=\lambda^2\int_0^{\infty} \sum_x (x_i-j'_i(\rho) t)^2 S(x,t) e^{-\lambda t} dt=C_{ii}+
2\chi \la\la w_i,(\lambda-L)^{-1}w_i\ra\ra-2\chi \la\la v_i,(\lambda-L)^{-1} v_i\ra\ra,
\end{equation}
where $
C_{ii}=\sum_y y_i^2 \, E \, 
r(y,  \eta)$,  for local functions $f, g$,
\begin{align}\label{singledouble}
\ll f, g\rr:=\sum_x \l f; \tau_x g\r=\lim_{\ell\to \infty} \frac{1}{\ell} \l \sum_{i=0}^{\ell-1} \tau_i f ; \sum_{i=0}^{\ell-1} \tau_i g\r
\end{align}
and, for $\mathcal P_1$, the orthogonal projection to the linear subspace generated by the functions  $\eta_i$ we define
\begin{equation}
w_i=\frac{1}{2}(W_i+W_i^*-\mathcal P_1 W_i-\mathcal P_1 W^*), \qquad v_i=\frac{1}{2}(W_i-W_i^*).\label{newflux}
\end{equation} 

 Note that in the case of finite range exclusion processes,
the third term on the right in equations  (\ref{taub}) vanishes. This is because in that case $v_i=\nabla_{e_i} q$ with $q$ being a first degree polynomial and thus the corresponding scalar product is zero. For the general case we will show that there is a model dependent constant $C<\infty$ such that \begin{equation}\label{l:vbnd2}
\ll v_i,(\lambda-L)^{-1}v_i\rr\le C.
\end{equation}
Hence our results are equivalent to statements about the blowup as $\lambda\searrow 0$ of 
$\ll w_i,(\lambda-L)^{-1}w_i\rr$.

We are interested in the following predictions for the asymptotics for the diffusivity as $\lambda\searrow 0$.  We say that a model is \emph{diffusive} (in the $i^{th}$ coordinate) if ${\hat D}_{ii}(\lambda)$ has a finite limit and superdiffusive (in the $i^{th}$ coordinate) if ${\hat D}_{ii}(\lambda)\to \infty$. The predictions make precise
in this context the arguments of the introduction:

\noindent\underline{Dimension $d=1$}
\begin{enumerate}
\item[(a)] If {$j''(\rho)\neq 0$} then the expected behavior is  superdiffusive with ${\hat D}(\lambda)\sim \lambda^{-1/3}$.
\item[(b)] If  $j''(\rho)=0, j'''(\rho)\neq 0$ then the expected behavior is  superdiffusive with  ${\hat D}(\lambda)\sim (\log \lambda^{-1})^{1/2}$.
\item[(c)] If $j''(\rho)=j'''(\rho)=0$ then the model is expected to be diffusive. 
\end{enumerate}
\noindent\underline{Dimension $d=2$}
\begin{enumerate}
\item[(d)] If  $j_{1}''(\rho)\neq 0$ then the expected behavior is  superdiffusive with ${\hat D}_{11}(\lambda)\sim (\log \lambda^{-1})^{2/3}$. 
\item[(e)] If  $j_1''(\rho)=0$ then the model is expected to be diffusive.
\end{enumerate}
\noindent\underline{Dimension $d\ge 3$} The model is expected to be diffusive with ${\hat D}(\lambda)$ converging to a finite limit. \medskip

%

Our results say basically that if the model is expected to be diffusive, it is indeed diffusive, and if it is expected to be superdiffusive, it is indeed superdiffusive.  Note that the constants in the statements may depend on the density $\rho\in(0,1)$.

\begin{theorem}\label{thm:1d} Assume that the speed change model has local, coercive rates  satisfying the divergence condition.  Then

\noindent \underline{Dimension $d=1$}
\begin{enumerate}
\item[(a)] If $j''(\rho)\neq 0$ then 
\begin{equation}\label{lbd1j}
C_1 \lambda^{-1/4}\le {\hat D}(\lambda) \le C_2 \lambda^{-1/2}
\end{equation}
\item[(b)]  If $j''(\rho)=0$ and $j'''(\rho)\neq 0$ then 
\begin{equation}
C_1 \log \log \lambda^{-1} \le {\hat D}(\lambda) \le  C_2 \log \lambda^{-1}
\end{equation}
\item[(c)] 
If $j''(\rho)=j'''(\rho)=0$ then ${\hat D}_{11}(\lambda) \le  C$. 
\end{enumerate}
\noindent \underline{Dimension $d=2$}
\begin{enumerate}
\item[(a)] If $j_1''(\rho)\neq 0$ then we have 
\begin{equation}
C_1 (\log \lambda^{-1})^{1/2}\le  {\hat D}_{11}(\lambda) \le  C_2 \log \lambda^{-1}
\end{equation}

\item[(b)] If $j_1''(\rho)=0$  then $  {\hat D}_{11}(\lambda) \le  C$.

\end{enumerate}
\noindent \underline{Dimension $d\ge 3$}

In all cases $ {\hat D}_{ii}(\lambda) \le  C$.

\end{theorem}

\begin{remark} The $d=1$, $j''(\rho)=0$, $j'''(\rho)\neq 0$ case bears many similarities in scaling of the
diffusivity to the $d=2$ lattice gas models for Navier-Stokes studied in \cite{landim-ramirez-yau} and certain random diffusions in  random environment in $d=2$ \cite{toth-valko}.  
This is not completely coincidental as there are some structural similarities between models in $d=1$ at 
inflection points and models in $d=2$ without preferred direction which make the both superdiffusivity and
the estimates based on degree $3$ test functions coincide in the two cases.  However, the technical issues involved
in the  proofs are very different, and the test functions themselves are different, although  the basic strategy based on variational methods is the same.  

The key issues in the present proofs are (i)  the identification of the leading order term (see (\ref{flux_der})),
and (ii)  to show that all the many other pieces of the asymmetric part of the generator do not play a role.
In the generic $d=1$, $j''(\rho)\neq 0$ case, for example, the leading order terms are the same as
for asymmetric exclusion, and the eventual test function used for the lower bound (\ref{lbd1j}) is the same as
in earlier work \cite{LQSY}.  What is new here in that case is (ii), i.e. the {\it universality}.  

In earlier work \cite{QV1, QV2}, we proved the universality of the $t^{1/3}$ law among one dimensional exclusion processes with drift.  This was possible because in that case we have a strong version of (ii), as well as a special solvable model (TASEP) within the class.  Unfortunately, we are not aware of any solvable model in the 
$d=1$, $j''(\rho)=0$, $j'''(\rho)\neq 0$ class, so at the present time we have no way to obtain the exact asymptotics
$(\log t)^{1/2}$.  Our lower bound is via a degree $3$ test function, which is complicated, and in fact we only
have identified its Fourier transform.   In the $d=2$, generic case of $j''_i(\rho)\neq 0$ for $i=1$ or $2$, one does have a type of  exact result for the the special nearest
neighbour model which is symmetric in one coordinate direction and totally asymmetric in the other coordinate
direction, at $\rho=1/2$ \cite{Yau2004} .  In principle, one expects that with sufficient effort, one could prove a strong enough 
result of type (ii) above, to show that this is universal among such models.  We have not pursued this here.  But note
that even extending the results of \cite{Yau2004} to $\rho\neq 1/2$ is non-trivial as the diagonal part of the asymmetric 
part of the generator no longer vanishes, and is not straightforward to control.  

Note that the full predictions for the superdiffusivities
can be obtained formally from the variational method if one assumes certain scaling properties of the 
optimal test function, and (more severely) that all off diagonal terms in the computations vanish.  The 
formal computations are described in \cite{landim-ramirez-yau, Herbert, toth-valko}.   Since they are fairly analogous in 
our case, we do not repeat them here.

We note that there are many ways to measure diffusivity; we concentrate on $\hat{D}(\lambda)$
here because our method is basically to study the blowup of the second term on the right hand side of (\ref{taub}).  One can certainly study the blowup of the resolvent for functions other than the current, but 
the current seems most natural from our point of view, and through the Green-Kubo formula is connected
directly to the second moment of the correlation functions.  From the Laplace transform to this second moment
seems to us a technical issue about regularity, and it is hard to say whether one should prefer one to the other.
 In one dimension, one also has a height function $h(t,x)$, whose discrete derivative 
is given by $\eta(t,x)$.   Part (b) in Theorem \ref{thm:1d} is roughly equivalent  to the conjecture that if $j''(\rho)=0, j'''(\rho)\neq 0$ then 
$h(t,0) \sim ct + t^{1/4}(\log t)^{1/8} \zeta$.  The fluctuation $\zeta$ is predicted to be Gaussian.  Note that this is not related to anisotropic KPZ in $d=2$, which is proved to have $h(t,0) \sim ct + (\log t)^{1/2} \zeta$ or isotropic KPZ in $d=2$ which is conjectured to have $h(t,0) \sim ct + t^{0.24\ldots} \zeta$.  In order to study a fluctuation like $t^{1/4}(\log t)^{1/8}$ one would need a different method than the present one.  The problem is that only objects whose time correlations decay slower than $t^{-1}$  will have a diverging resolvent.  Of course, a 
natural thing to try would be some quadratic form like $\la w,(\lambda-L)^{-2} w\ra$ or perhaps a higher power $\la w,(\lambda-L)^{-p} w\ra$, depending 
on the context. However, recall L\"owner's theorem, which says that a function $f$ preserves the definiteness of 
$\lambda-L$ if and only if it is real analytic and has  analytic continuations to the upper and lower half planes which maps  the upper half plane to itself.  Among powers this singles out $\la w,(\lambda-L)^{-1} w\ra$.
\end{remark}

To understand the microscopic origin of the conditions on $j''(\rho)$ and $j'''(\rho)$ we  study the representation of functions in $L_2(\{0,1\}^d, \pi_\rho)$  in the following orthonormal basis. For any finite $\Lambda\subset \Z^d$, let
\begin{equation}
\label{heta}\heta_\Lambda=\heta^\rho_\Lambda=\prod_{x\in \Lambda} \chi^{-1/2} (\eta_x-\rho),
\end{equation}
where $\chi=\rho(1-\rho)$.    Let $\mathcal M_n$ denote the local functions of degree $n$, i.e.~$f$ is a finite sum $ f=\sum_{|\Lambda|=n}f_\Lambda \heta_\Lambda$.  We also denote by  
$\bar{\mathcal M}_n$ the
closure of this space in $L^2(P)$, i.e. 
\begin{equation}
f\in  \bar{\mathcal M}_n \qquad \Leftrightarrow\qquad f=\sum_{|\Lambda|=n}f_\Lambda \heta_\Lambda, \quad 
\sum_{|\Lambda|=n}f^2_\Lambda<\infty.
\end{equation}For $f, g\in \bar{\cM}_n$ we have
\begin{align}
\l f, g\r =\sum_{|\Lambda|=n} f_\Lambda g_\Lambda.
\end{align}
If $f\in \bar{\cM}_n$, $g\in \bar{\cM}_m$, $m\neq n$ then we have $\l f, g\r=0$, so our Hilbert space is naturally graded \begin{equation}L_2(\{0,1\}^{\Z^d}, \pi_\rho)=\oplus_{n=0}^\infty \bar{\mathcal{M}}_n. \end{equation}
For the $\ll, \rr$ scalar product of $f, g \in \cM_n$ (defined in (\ref{singledouble})) we have
\begin{equation}\label{scproddred}
\ll f, g\rr= \sum_{|\Lambda|=n, y\in \Z^d} f_\Lambda g_{\Lambda+y}=\sum_{\Lambda\in \mathcal B_n[d]} \bar f_\Lambda \bar g_\Lambda, \qquad \bar f_\Lambda=\sum_{y\in\Z^{d}} f_{\Lambda+y},
\end{equation}
where $\mathcal B_n[d]$ is the set of equivalence classes of size $n$ subsets of $\Z^d$ modulo the translation. (I.e. $\Lambda\sim \Lambda'$ if $\Lambda'=\Lambda+y$.) We call these {\em dimension reduced variables}.
If $d=1$ then we can represent an element of $\mathcal B_n[1]$ as an element of $\Z_+^{n-1}$: the equivalence class of the ordered $n$-tuple $(x_1,x_2,\dots, x_n)$ is  represented by the vector $(x_2-x_1-1, x_3-x_2-1, \dots, x_n-x_{n-1}-1)$.
In that case $\bar f$ can be represented as a function on $\Z_+^{n-1}$ and the scalar product is the usual:
\begin{equation}
\ll f, g \rr=\l \bar f, \bar g\r=\sum_{x\in \Z_+^{n-1}} \bar f(x) \bar g(x).
\end{equation}
Now suppose that (dropping the subscript $i$ for clarity), \begin{equation} \label{doubleu}
w=\sum_\Lambda w_\Lambda^\rho \heta_\Lambda^\rho\end{equation}
is the representation of $i^{{th}}$ coordinate of the symmetrized microscopic flux. Note that $\tilde j_i(\rho):=E_{\pi_\rho} w_i$ and $j_i(\rho)=E_{\pi_\rho} W_i$ only differ in a linear term so their second and higher order derivatives will be the same. In fact, with a slight abuse of notation, we will drop the tilde in $\tilde j$ from this point.
Since $E_{\rho+\varepsilon}\heta_x^\rho=\chi^{-1/2}\varepsilon$ we get
\begin{equation}
j_i(\rho+\varepsilon)=\sum_\Lambda w_\Lambda^\rho E_{\rho+\varepsilon} \heta_\Lambda^\rho= \sum_\Lambda w_\Lambda^\rho \varepsilon^{|\Lambda|}\chi^{-|\Lambda|/2}
\end{equation}
and hence 
\begin{equation}{\partial_\rho^k}j_i(\rho)=k! \chi^{-k/2} \sum_{|\Lambda|=k} w_\Lambda^\rho.\label{flux_der}\end{equation} 
Note that $w_i$ is always a function of degree $2$ and higher.  Hence,
\begin{equation}\label{flux_der2}
j''(\rho)=0\qquad \Leftrightarrow\qquad \sum_{|\Lambda|=2} w_\Lambda^\rho=0
\end{equation}
and
\begin{equation}\label{flux_der3}
j'''(\rho)=0\qquad \Leftrightarrow\qquad \sum_{|\Lambda|=3} w_\Lambda^\rho=0.
\end{equation}
The problem becomes one of describing the behaviour of $\ll w, (\lambda-L)^{-1} w \rr$ as $\lambda\searrow 0$ according to conditions such as the right hand sides of  (\ref{flux_der2}) and (\ref{flux_der3}).

We study $\ll w, (\lambda-L)^{-1} w \rr$ using the variational formula, which we learned from S.R.S.~Varadhan,
\begin{equation}\label{var_form}
\ll w, (\lambda-L)^{-1} w\rr=\sup_{f} \left\{2 \ll w, f \rr-\ll f, (\lambda-S) f\rr-\ll A f,(\lambda-S)^{-1} Af\rr\right\}
\end{equation}
where the supremum is taken over local functions. The operator $A=\frac{L-L^*}{2}$ is the asymmetric part of the generator. 
An easy upper bound can be obtained by dropping the final, non-negative, term and using  (\ref{S_comp}),
\begin{equation}\label{H-1upper}
\ll w, (\lambda-L)^{-1} w\rr\le \sup_{f} \left\{2 \ll w, f \rr-\ll f, (\lambda-S) f\rr\right\}=
\ll w, (\lambda-S)^{-1} w\rr\le C\ll w, (\lambda-S_0)^{-1} w\rr,
\end{equation}
which is computable because $S_0$ acts diagonally on $\oplus_{n=0}^\infty \mathcal{M}_n$ (see the discussion below and also Subsection \ref{subs:sym}).
In the other direction we can use the bound (\ref{S_comp})  to obtain
\begin{align}
\ll w, (\lambda-L)^{-1} w\rr\ge C \sup_{f} \left\{2 \ll w, f \rr-\ll f, (\lambda-{\SSEP}) f\rr-\ll A f,(\lambda-{\SSEP})^{-1} Af\rr\right\}. \label{newvarpr}
\end{align}
This is still not computable because the operators  $A$ map $\mathcal{M}_n$ to a (finite) direct sum
$\oplus_{|k|\le \ell} \mathcal{M}_{n+k}$. 
Recall that  $\SSEP$ denotes the generator of the symmetric simple exclusion (we do not denote the dependence on $d$).
We note that if $f=\sum_\Lambda f_\Lambda \heta_\Lambda$ then
\begin{equation}
\SSEP f(\eta)=\sum_{x,e_i} f(\eta^{x,x+e_i})-f(\eta)=\sum_\Lambda \heta_\Lambda \sum_{|\{x,x+e_i\}\cap\Lambda|=1} \left( f_{\Lambda^{x,x+e_i}}-f_\Lambda\right).
\end{equation}
Here $\Lambda^{x,x+e_i}$ is the finite set one obtains by replacing $x$ with $x+e_i$ or $x+e_i$ with $x$ in $\Lambda$. In particular, $\Lambda=\Lambda^{x,x+e_i}$ unless $\left|\Lambda\cap \{x,x+e_i\}\right|=1$.  $\SSEP$ maps the degree $n$ functions, $\mathcal M_n$ into itself. 
If we think about the coefficients $\{f_\Lambda\}$ as a function on finite particle configuration $|\Lambda|=n$ then $\SSEP$ is just the generator of a  collection of $n$ particles moving according to continuous time simple symmetric random walk with the exclusion rule. This is the duality property of the symmetric simple exclusion process \cite{liggett}.

In one dimension $\SSEP$ acts on the dimension reduced form of $f$ the following way: \begin{equation}\label{symgen}
\SSEP f(y_1,\dots, y_{n-1})=\sum_{i=1}^{n} \nabla^+_i f(\und y)+\nabla^-_i f(\und y)
\end{equation}
where $\nabla^{\pm}_i$ is the effect of moving the $i^{th}$ particle to the right or  left (if there is space):
\begin{eqnarray*}
\nabla^+_i f(\und y)&=&\ind(y_{i}>0)(f(\dots,y_{i-1}+1,y_{i}-1,\dots)-f(\und y))\\
\nabla^-_i f(\und y)&=&\ind(y_{i-1}>0)(f(\dots,y_{i-1}-1,y_{i}+1,\dots)-f(\und y))
\end{eqnarray*}
where $y_0=y_n=\infty$. 
Note that this is a symmetric random walk in $\Z_+^{n-1}$ where the possible moves are 
$
\pm e_1, \pm(e_2-e_1), \dots, \pm(e_{n-1}-e_{n-2}), \pm e_{n-1}
$
and each can happen with rate one (if the corresponding move is allowed, i.e. we stay in $\Z_+^{n-1}$). The Dirichlet form in the dimension reduced form is given by
\begin{align}\label{dir_excl}
\la f, (-\SSEP), f \ra=\frac12 \sum_{\und y} \sum_{i=1}^n \left(\nabla^{+}_if(\und y)\right)^2.
\end{align}

In any dimension $d$, $\SSEP$ operates diagonally on the direct sum of the $\mathcal M_n$, and in each is 
just a Laplace operator on the antisymmetric subspace of $(\Z^d)^n$.  The variational problem involves the 
operators $A$ which couple the different $\mathcal M_n$.  Computing the action of $A$ from $\mathcal M_n$
to $\mathcal M_{n+k}$, $k\ge 1$, it takes appropriate differences of a function and 
employs them as various boundary values for the part of
the variational problem in $\mathcal M_{n+k}$.  Thus the problem can be represented as  a set of variational problems on positive quadrants of $(\Z^d)^n$, with  the different $n$s coupled through boundary conditions.  Even 
a cutoff version of the problem, optimising over test functions of degree two or three, turns out to be challenging.

%
 
  If one restricts the supremum in (\ref{var_form}) to functions
of degree less than or equal to $n$, one of course obtains an increasing sequence of  lower bounds.  On the other hand, the last term
$\ll A f,(\lambda-S)^{-1} Af\rr$ in (\ref{var_form}) is itself given by a variational formula, and one can restrict
the test functions there to obtain a decreasing sequence of upper bounds.
In the case of the exclusion process at $\rho=1/2$, slightly tweaking this reasoning by simply solving an approximate resolvent equation restricted to functions of degree less than or equal to $n$ one can obtain 
an oscillating sequence of upper and lower bounds similar to (\ref{H-1upper}) and
(\ref{newvarpr}) via a continued fraction-like approximation of $(\lambda-L)^{-1}$ (see \cite{LQSY} for details). This relies on the fact that the generator in this case is tridiagonal on the graded space and that the diagonal part is symmetric.  For the general speed change generators the upper and lower bounds no longer hold.  So these estimates are not available in our context.   In addition, in the two special cases $d=1$, $j''(0)\neq 0$ and
$d=1$, $j_i''(0)=0$, the sequence of upper and lower bounds seem to be {\it very slowly convergent}.  The
strange case is $d=2$, $j_i''(0)\neq 0$ which has logarithmically diverging diffusivity, but fast convergence
of these bounds.  In a sense, the case $d=1$, $j''(0)=0$ is the hardest of all, because of the slow convergence
coupled with the lack of a solvable model to compare to.

\section{Key tools and proofs of the main results}

\subsection{Estimates on the symmetric part of the generator}\label{subs:sym}

%

For the $d$-dimensional SSEP we have $\SSEP=\sum_{i=1}^n \sum_{j=1}^d (\nabla_{i,-e_j}+\nabla_{i,e_j})$ where $\nabla_{i,\pm e_i}$ is the effect of moving the $i^{th}$ particle in the direction $\pm e_j$ (if possible). The formula for $\la f, (-\SSEP) f \ra$ (in the dimension reduced form) is similar to (\ref{dir_excl}): we just have to sum $(\nabla_{i,\pm e_i} f)^2$ on $\mathcal B_n[d]$ and $i=1, \dots, d$.


%
%
%

\begin{lemma}\label{l:harmonic} 
Let $f:\mathcal B_n[d]\to \R$ be a fixed function and $\Lambda\in \mathcal B_n[d]$. Then 
\begin{align}\label{1st symbnd}
f(\Lambda)^2\le C_n(\lambda) \la f, (-\SSEP)f \ra
\end{align}
where there is a $C=C(d)<\infty$ such that 
\begin{align}
C_n(\lambda)=\begin{cases}
C  \qquad &\textup{ if $n\ge 4$ and $d=1$, or if $n\ge 3$ and $d=2$ or if $n\ge 2, d\ge 3$},\\
 C |\log \lambda| \qquad& \textup{ if $n=3, d=1$ or $n=2, d=2$},\\
 C \lambda^{-1/2} \qquad &\textup{ if $n=2, d=1$}.
\end{cases}
\end{align}
%
%
%
\end{lemma}
\begin{proof}  We need to consider the dimensions case by case:

\underline{$d=1$.}   As we have already discussed, we may represent $\mathcal B_n[1]$ as $\Z_+^{n-1}$.
In order to show (\ref{1st symbnd}) for $n\ge 4, d=1$ we may assume that $n=4$. Indeed, we may consider the function $\tilde f:\Z^3_+\to \R$ with $\tilde f(x_1,x_2,x_3): =f(x_1,x_2,x_3,y_4,\dots, y_{n-1})$ for some fixed $y_4, \dots, y_{n-1}$ and apply the $n=4$ case for $\tilde f$ and $\tilde x=(x_1,x_2,x_3)$. The only thing we have to check 
is that $\l \tilde f, (-S), \tilde f\r$ can be bounded by $ C \la f, (-S)f \ra$ with a fixed $C$. But this is clear from  (\ref{dir_excl}) and
\begin{align}
\nabla^+_4\tilde f(x_1,x_2,x_3)=-\nabla^+_1 f(x_1+1,x_2,x_3,\dots)-\nabla^+_2 f(x_1,x_2+1,x_3,\dots)-\nabla^+_3 f(x_1,x_2,x_3+1,\dots),
\end{align}
which just states that the we can increase the third gap among $n$ particles by moving the first three to the left by one step.

To prove (\ref{1st symbnd}) we need to show that 
\begin{align}
\inf_{f,y}  \frac{1}{f(y)^2}  \la f,(-\SSEP) f \ra>0.
\end{align}
where $f: \Z_+^3\to \R$, $y\in \Z_+^3$. 
Let $\tilde f$ be the extension of $f$ to $\Z^3$ by reflections: $\tilde f(\pm y_1,\pm y_2, \pm y_3):=f(y_1,y_2,y_3)$. Then $8\la f,(-\SSEP) f \ra\ge  \la \tilde f, (-\tilde S) \tilde f\ra$ where $\tilde S$ is the generator of the same random walk, but now on $\Z^3$. The new variational problem can be solved exactly by Fourier transform:
\begin{align}
\inf_g  \frac{1}{g(y)^2}  \la g,(-\tilde S) g \ra> C \left(\int_{[-\pi,\pi]^3} \frac{1}{t_1^2+(t_2-t_1)^2+(t_3-t_2)^2+t_3^2}dt_1dt_2 dt_3\right)^{-1}>0.
\end{align}
For the $n=3$ the same argument gives
\begin{align}
\inf_g  \frac{1}{g(y)^2}  \la g,(\lambda-\tilde S) g \ra\ge C \left(\int_{[-\pi,\pi]^2} \frac{1}{\lambda+t_1^2+(t_2-t_1)^2+t_2^2}dt_1dt_2\right)^{-1}\ge C |\log \lambda|^{-1}.
\end{align}
for the $n=2$ case one gets the bound 
\begin{align}
\inf_g  \frac{1}{g(y)^2}  \la g,(\lambda-\tilde S) g \ra\ge C \left(\int_{[-\pi,\pi]^2} \frac{1}{\lambda+t^2}dt\right)^{-1}\ge C  \lambda^{-1/2}.
\end{align}

\underline{$d=2$.} As in the previous case, we may assume that $n=3$. We may also assume that $\Lambda=\{x_1,x_2,x_3\}$ has three elements with different $x$ coordinates (which are ordered according to the index). Otherwise we could move two of the particles by one unit, and we could bound the difference with a constant times  $\l f, (-\SSEP) f\r$. Now consider a new generator $\tilde S$ where $x_1$ and $x_2$ performs symmetric random walks with the constraint that the order of the $x$-coordinates of the three particles is  preserved. Clearly $\l f (-\SSEP) f \r\ge \l f (-\tilde S) f \r$ and $\tilde S$ corresponds to a symmetric simple random walk on $(\Z_+\times \Z)^2$. The proof now can be finished similarly as in the $n=1$ case.

For $n=2$ the proof is similar, we may consider the generator of the simple symmetric random walk on $\Z_+ \times \Z$ instead of $S$ and after the reflection trick (see the $d=1$ proof) the variational problem can be computed explicitly.

\underline{$d\ge 3$.} The proof goes along the lines of the previous arguments: we may assume $n=2$ and replace $S$ with the generator of the simple symmetric random walk on $\Z_+\times \Z^{d-1}$.
\end{proof}
%
%

\begin{lemma}\label{l:Sbounds} 1. ($H_{-1}$ bounds for homogeneous polynomials.)
\begin{align}\label{1dH-1}&\underline{d=1}\qquad \qquad \qquad 
\ll \heta_\Lambda, (\lambda-\SSEP)^{-1} \heta_\Lambda\rr \le \begin{cases}
C \lambda^{-1/2}&\qquad \textup{ if } |\Lambda|=2,\\
C |\log \lambda|&\qquad \textup{ if } |\Lambda|=3,\\
C&\qquad \textup{ if } |\Lambda|>3;\end{cases}\\ 
&\underline{d=2}
\qquad \qquad \qquad 
\ll \heta_\Lambda, (\lambda-\SSEP)^{-1} \heta_\Lambda\rr \le \begin{cases}
C |\log \lambda|&\qquad \textup{ if } |\Lambda|=2,\\
C &\qquad \textup{ if } |\Lambda|>2
\end{cases}
\\
 &\underline{d\ge3}\qquad \qquad \qquad 
\ll \heta_\Lambda, (\lambda-\SSEP)^{-1} \heta_\Lambda\rr \le C
\end{align}
2. (Triviality criterion) For all $d,k\ge 1$, let  $c_{\Lambda}, \Lambda\subset \Z^d, |\Lambda|=k$ be a finite collection of coefficients with  $\sum_{\Lambda} c_\Lambda=0$. Then there is a $C$ depending on the $c_\Lambda$ such that
\begin{equation}\label{gradbound}
\ll \sum_{\Lambda} c_\Lambda \heta_{\Lambda},(\lambda-\SSEP)^{-1}   \sum_{\Lambda} c_\Lambda \heta_{\Lambda}   \rr\le C.
\end{equation}

\end{lemma}

%
%

\begin{proof}
By the variational formula for any given $\alpha>0$ we have
\begin{align}
\ll \heta_\Lambda, (\lambda-\SSEP)^{-1} \heta_\Lambda \rr=\sup_f 2\alpha \ll \heta_\Lambda, f\rr-\alpha^2 \ll f, (\lambda-\SSEP) f \rr.
\end{align}
Thus if we have a uniform bound  $\ll \heta_\Lambda, f\rr\le C \ll f, (-\SSEP) f\rr$ then 
$
\ll \heta_\Lambda, (\lambda-\SSEP)^{-1} \heta_\Lambda \rr\le 2 C \alpha-\alpha^2
$
which gives an upper bound $= C^2$ with the choice $\alpha=C$. From this most of the statements follow from the previous lemma, we only need to prove  (\ref{gradbound}). 

If  $\sum_{\Lambda} c_\Lambda=0$ and all subsets $\Lambda$ are the same size then we can represent $f=\sum_{\Lambda} c_\Lambda \heta_\Lambda$ as a finite linear combination of gradients $\heta_\Omega-\heta_{\tilde\Omega}$ where $|\Omega|=|\tilde \Omega|=k$ and $\tilde \Omega=\Omega^{x,y}$ where $x,y$  are nearest neighbors. (This is clear if $f$ is of the form $\heta_{\Lambda_1}-\heta_{\Lambda_2}$ and the general case follows by linearity.)
We will show that 
\begin{align*}
\ll \heta_\Omega-\heta_{\tilde \Omega}, (\lambda-\SSEP)^{-1} (\heta_\Omega-\heta_{\tilde \Omega})\rr\le 1,
\end{align*}
from this the  (\ref{gradbound}) follows immediately. Using the dimension reduced variables in $\mathcal B_k[d]$ and the variational formula we get
\begin{align}
\ll \heta_\Omega-\heta_{\tilde \Omega}, (\lambda-\SSEP)^{-1} (\heta_\Omega-\heta_{\tilde \Omega})\rr&=
\sup_g 2(g_\Lambda-g_{\tilde \Lambda})-\l g, (\lambda-\SSEP) g\r 
\end{align}
But  $\l g, (\lambda-\SSEP) g\r\ge(g_\Lambda-g_{\tilde \Lambda})^2 $ so the right hand side is bounded, which completes the proof.
%
%
%
%
%
%
\end{proof}

\subsection{Reduction of the asymmetric part of the generator}

For the lower bound computations we will only need the evaluate the asymmetric part of the generator on  degree two or three functions (before dimension reduction). In this section we will show how we can simplify these operators in the general case. We will work with the generator in the dual representation, and in the one dimensional cases we also use dimension reduction.


 We review the simplest case of the one dimensional TASEP  \cite{LQSY} here, before going on to discuss the general speed change model. Although the process is totally asymmetric, the generator $L$ is the sum of a symmetric
 part $\SSEP=\tfrac{1}{2}(L+L^*)$  (\ref{symgen}) and an asymmetric part $A_{TASEP}=\tfrac{1}{2}(L-L^*)$.
  The latter can be expressed as a sum $A_0+A_++A_{-}$ where $A_+$ maps a degree $n$ function into a degree $n+1$ function, $A_0=-A_0^*$ preserves the degree and $A_{-}=-A_+^*$ decreases the degree by one. They are written in duality after dimension reduction as
\begin{align}\notag
(A_0f)(y_1, \dots, y_n)=&\frac{1-2\rho}{2} \sum_{i=1}^{n} \ind(y_{i-1}>0) (f(\dots, y_{i-1}-1,y_{i}+1,\dots)-f(\und{y}))\\
&-\frac{1-2\rho}{2} \sum_{i=1}^{n} \ind(y_{i}>0) (f(\dots, y_{i-1}+1,y_{i}-1,\dots)-f(\und{y}))\label{A0TASEP};
\end{align}
%

%
\begin{equation}\label{A+TASEP}
(A_+f)(y_1,\dots,y_n)=\sqrt{\chi} \sum_{i=0}^n \ind(y_i=0)(f(\dots, y_{i-1}+1,y_{i+1}, \dots)-f(\dots, y_{i-1}, y_{i+1}+1,\dots));
\end{equation}
%
\begin{equation}\label{A-TASEP}
(A_-f)(y_1,\dots,y_n)=\sqrt{\chi} \sum_{i=1}^n \ind(y_i>0)(f(\dots, y_{i-1}+1,0,y_{i}, \dots)-f(\dots, y_{i-1}, 0,y_{i}-1,\dots)).
\end{equation} 
%
In all of these formulas  $y_0=y_{n+1}=\infty$.


Our lower bounds will rely on finite degree test function computations. More specifically, we will use the bound (\ref{newvarpr}) and restrict the supremum to certain finite degree test functions. Recall the definition of $K$ from the assumptions on our rate function $r(\cdot,\cdot)$.

In case of the  general one dimensional model we will use
\begin{align}\label{M21}
\cM_2^{(1)}:=\left\{ f\in \cM_2:  \bar f(x)=\bar f(y) \textup{ if } \left\lfloor\frac{x}{10 K} \right\rfloor=\left\lfloor\frac{y}{10 K} \right\rfloor \textup{ for } x, y\in \Z_+ \right\}.
\end{align}
For the one dimensional model at inflection we consider
\begin{align}\label{M31}
\cM_3^{(1)}:=&\left\{ f\in \cM_3:   \bar f(x,y)=\bar f(y,x) \textup{ for } x,y\in \Z_+ \textup{ and }\right. \\  &\quad \left. 
\bar f(x_1,x_2)=\bar  f(y_1,y_2) \textup{ if } \left\lfloor\frac{x_i}{10 K} \right\rfloor=\left\lfloor\frac{y_i}{10 K} \right\rfloor \textup{ for } i=1,2, x_i, y_i\in \Z_+, 
\bar f(x,y)=\bar f(y,x) \right\}.\notag
\end{align}
Finally, for the general two dimensional model we will work on
\begin{align}\label{M22}
\cM_2^{(2)}:=\left\{   f\in \cM_2:
\sum_{x\in \Z^2} f_{\Lambda+x}=\sum_{x\in \Z^2} f_{\{0,1\}+x},  \textup{ for every } |\Lambda|=2, \Lambda\subset [0,10 K]^2
\right\}.
\end{align}
Lemma \ref{l:Abnd1d} below  describes how $A$ simplifies in the general one dimensional cases if we restrict it to the specific classes of test functions. It shows that $A$ can be approximated with a constant multiple of $A_++A_{-}$ from the TASEP generator (see (\ref{A+TASEP}) and (\ref{A-TASEP})) and the  linear combination of certain simple operators. The key point is that at the inflection point the TASEP part of the generator vanishes. 

In order to deal with the negligible terms we introduce the following definition. We say that an operator $T$ is \emph{sectorial on  $\mathcal G$} if there is a $C<\infty$ such that for any $f\in \mathcal G$,
\begin{align}
\ll T f, (\lambda-\SSEP)^{-1} Tf\rr \le C \ll f, (\lambda-\SSEP)^{-1} f\rr. 
\end{align}
The idea is that any sectorial part of the operator $A$ can be dropped from the right hand side of (\ref{newvarpr}) without affecting the quality of the lower bound.  Our dream was that $ A+\frac{j''(\rho)}{2}A_{TASEP}$ would be sectorial\footnote{Recall the for TASEP $j''(\rho)=-2$, hence this is really a difference}.  This is unfortunately false.
What is true is that one can find a reasonably simple operator $\tilde{A}$ so that $ A+\frac{j''(\rho)}{2}A_{TASEP}-\tilde A$ satisfies a graded version of the sector condition, i.e. it is sectorial when restricted 
to a finite degree subspace.  Since the bounds one obtains  improve very slowly with degree, one might as well
restrict to the lowest degree where they give something non-trivial, which turns out to be degree $3$ in
$d=1$.     The operator $\tilde{A}$ is not unique, and we only really care about how it acts on $f\in \cM^{(1)}_2\cup \cM^{(1)}_3$.  We will prove that the following choice will work:  In  dimension reduced representation,
if \underline{ $f\in \cM^{(1)}_2$ } then
\begin{align}
\tilde  A f(y_1,y_2,y_3)=&C( \indd{y_1=y_2=0}(f(y_3+1)-f(y_3))+\indd{y_2=y_3=0}(f(y_1+1)-f(y_1+2)));\label{A24}
\end{align}
If \underline{ $f\in \cM^{(1)}_3$}  we have $\tilde A=c_1 \tilde A_1+c_2 \tilde A_2+c_3 \tilde A_3$ with 
\begin{align} \label{A33}
\tilde A_1 f(y_1,y_2)&= \indd{y_1=0}(f(0,y_2)-f(0,y_2+1))+\indd{y_2=0}(f(y_1,0)-f(y_1+1,0))\\
\tilde A_2 f(y_1,y_2,y_3)&=\indd{y_1=y_2=0}(f(0,y_3+1)-f(0,y_3))+\indd{y_2=y_3=0}(f(y_1,0)-f(y_1+1,0))\label{A34}\\
\notag
\tilde A_3 f(y_1,y_2,y_3,y_4)&=\indd{y_1=y_2=0}(f(y_3+1,y_4)-f(y_3,y_4))+\indd{y_2=y_3=0}(f(y_1+1,y_4+1)-f(y_1+2,y_4))\\
&\quad+\indd{y_3=y_4=0}(f(y_1,y_2+1)-f(y_1,y_2+2))\label{A35}
\end{align}
The following  lemmas are proved  in Section \ref{s:dual}.

\begin{lemma}[Simplified form of the asymmetric part in 1d]\label{l:Abnd1d} 
Consider a one dimensional coercive lattice gas with local rates satisfying the divergence condition, and denote the asymmetric part by $A$. Let $A_{TASEP}$ be the asymmetric part of the generator of TASEP for $\rho=1/2$. Then there is an  operator $\tilde A$  satisfying (\ref{A24}) - (\ref{A35})   so that \begin{align}\label{decomp}
 A+\frac{j''(\rho)}{2}A_{TASEP}-\tilde A
\end{align}
is sectorial on $\cM^{(1)}_2\cup \cM^{(1)}_3$.
\end{lemma}

\begin{lemma}\label{l:A2d}
Suppose that we have a lattice gas in two dimensions satisfying our conditions. Then $A-A_{TASEP}$ is sectorial on $\cM_2^{(2)}$  where $A_{TASEP}$ is the asymmetric part of the generator for a nearest neighbor 2d TASEP with drift $-2 (j_1''(\rho), j_2''(\rho))$.
\end{lemma}
Note that  it is easy to construct a nearest neighbor TASEP with a given drift $(a,b)$. Assuming $a, b\ge 0$ one can just take the jump law to be $p(e_1)=a+1, p(-e_1)=1, p(e_2)=b+1, p(-e_2)=1$ with $p(y)=0$ otherwise.

\subsection{Estimates allowing removal of the hard core} 

Although the previous lemmas show that $A$ can be simplified  on $\cM_2^{(1)}, \cM_3^{(1)}$ and $\cM_2^{(2)}$, even the simplified operators are too complicated  to do actual computations. The main reason is that because of the exclusion condition (the `hard core interaction') the test functions do not live on  $\Z^d$, but
on $\Z^d$ minus diagonals.  It is natural to try to remove the hard core interaction, i.e.~approximate the test functions, and operators, with ones on $\Z^d$. 

This was one of the main tools in \cite{LQSY} and \cite{Yau2004}. If one considers a test function in $\cM_n$ in dimension $d$ then one can naturally extend this to a symmetric $n$-variable function on $\Z^d$ which vanishes when two  coordinates coincide. The generators extend naturally (e.g.~$\SSEP$ will correspond to the lattice Laplacian) and one can use Fourier techniques to do computations with the new versions of these operators. The `removal of the hard core' lemmas of \cite{LQSY} (see Section 4 in that paper) show that when we make these approximations we do not make a large error when we replace $\ll f, (\lambda-\SSEP) f \rr$ and $\ll A_{TASEP} f, (\lambda-\SSEP)^{-1} A_{TASEP} f\rr$ with their approximate versions. If $d=2$ then the extra constants appearing in the estimates will only depend on $n$ (and not $\lambda$) while in the $d=1$ case an extra $|\log \lambda|$ multiplier appears. 

We are still able to use the $d=2$ version for our purposes, but the $|\log \lambda|$ factor would kill our estimates in the $d=1$ inflection case, because the order of the diffusivity is smaller, so we  use a different approach. 
We consider the dimension reduced form of our test functions to get an extension on $\Z$ (in case of $\cM_2^{(1)}$) and $\Z^2$ (in case of $\cM_3^{(1)}$) and show that the approximate versions of the appropriate quadratic forms give good enough bounds. These are the contents of Lemma \ref{1dhardcore} and Lemma \ref{1dhardcore_infl} below, whose proofs are given in  Section \ref{s:hardcore}.
\begin{lemma}\label{1dhardcore}
Consider a one dimensional lattice gas model satisfying our conditions. If $f:\Z_+\to \R$ is a test function with $f\in \cM_2^{(1)}$ 
then 
\begin{align}\label{hardcore11}
\l f, (\lambda-S) f \r &\le C \int_{[-\pi,\pi]} (\lambda+t^2) |\hat F(t)|^2 dt\\\label{hardcore12}
\l A_{TASEP}f, (\lambda-S)^{-1} A_{TASEP}f \r& \le C \int_{[-\pi, \pi]} \frac{t^2}{\sqrt{\lambda}}|\hat F(t)|^2 dt\\
\l \tilde A f, (\lambda-S)^{-1} \tilde A f\r&\le C \int_{[-\pi, \pi]} (\lambda+t^2+\frac{t^2}{\sqrt{\lambda}})|\hat F(t)|^2 dt .\label{hardcore13}
\end{align} 
Here $\tilde A$ is the operator from Lemma \ref{l:Abnd1d} and $\hat F(t)=\sum_{k=0}^\infty f(k)( e^{it k}+e^{i t (1-k)}) $ is the Fourier transform of the following symmetric extension of $f$:
\begin{align}
F(x)=f(|x|-\ind_{x<0})
\end{align}
\end{lemma}

\begin{lemma}\label{1dhardcore_infl}
Consider a one dimensional  lattice gas model satisfying our conditions.  If $f:\Z_+^2\to \R$ is a test function with $f\in \cM_3^{(1)}$
%
then the following bounds hold:
\begin{align}
\label{hc21}\l f, (\lambda-S) f \r &\le C \iint_{[-\pi,\pi]^2} (\lambda+t_1^2+t_2^2) |\hat F(t_1, t_2)|^2 dt_1 dt_2\\ \label{hc22}
\l \tilde A f, (\lambda-S)^{-1} \tilde A f\r&\le  C \iint_{[-\pi,\pi]^2}  (t_1^2+t_2^2) |\log \lambda | |\hat F(t_1, t_2)|^2dt_1dt_2\\
&+C \int_{[-\pi,\pi]} \left| \int_{[-\pi,\pi]}\hat F(t_1, t_2) dt_2\right|^2 |t_1| dt_1+C\l f, (\lambda-S) f \r.\notag
\end{align}
Here $\tilde A$ is the operator from Lemma \ref{l:Abnd1d} and $\hat F(t_1,t_2)$ is the Fourier transform of the following symmetric extension of $f$:
\begin{equation}\label{2dext}
F(x,y)=f(|x|-\ind_{x<0}, |y|-\ind_{y<0})
\end{equation}
\end{lemma}
Note the previous lemmas will be used for dimension reduced functions, and therefore 
are stated with only the single bracket.

\begin{lemma}\label{2dhardcore}
Consider a two dimensional finite range exclusion model with drift $a e_1+b e_2$.
 Let $f$ be a  test function in $\cM_2^{(2)}$ and assume that  $F:\Z^2\to \R$ is an extension of $f$,
\begin{align}\label{deF}
F(x_1,x_2)&=F(-x_1,-x_2)=\sum_{y_1,y_2\in \Z} f(\{(x_1+y_1,x_2+y_2),(y_1,y_2)\}), \quad \textup{ if } (x_1,x_2)\neq (0,0),\end{align}
with $F(x,x)=0$.
Then 
\begin{align}\label{2dhardc1}
\ll f, (\lambda-\SSEP) f\rr&\le C \iint_{[-\pi,\pi]^2} (\lambda+t_1^2+ t_2^2)| |\hat F(t_1,t_2)|^2 dt_1 dt_2
\\
\ll (A_++A_-) f,(\lambda-\SSEP)^{-1} (A_++A_-)  f\rr &\le C \iint_{[-\pi,\pi]^2} (a t_1+b t_2)^2|\log \lambda| |\hat F(t_1,t_2)|^2 dt_1 dt_2.\label{2dhardc2}
\end{align}
Here 
 $A_++A_-$ is the asymmetric part of the generator of the exclusion without the $A_0$ part.
\end{lemma}

%
%


\subsection{Proofs of the main results}\label{s:proofs}

\begin{proof}[Proof of the upper bounds in Theorem \ref{thm:1d}]
We begin by proving  (\ref{l:vbnd2}).
By  (\ref{H-1upper}) it is enough to prove   $\ll v_i, (\lambda-\SSEP)^{-1} v_i\rr\le C$.  By the definition of $v_i$ and $W_i$ it is clear that if we write $
v_i=\sum_{\Lambda} c_\Lambda \heta_\Lambda
$.
then for any fixed $k$ we have $\sum_{|\Lambda|=k} c_\Lambda=0$. Hence  (\ref{l:vbnd2}) follows from (\ref{gradbound}).  The rest of the upper bounds follow
immediately from (\ref{flux_der2}), (\ref{flux_der3}) and the two parts of Lemma \ref{l:Sbounds}.
\end{proof}
%
%
%
%
%


\begin{proof}[Proof of the lower bound in Theorem \ref{thm:1d} in $d=1$ when $j''(\rho)\neq 0$] 
We will evaluate the right hand side of (\ref{newvarpr}) with an appropriately chosen $f\in \cM_2^{(1)}$. By the Cauchy-Schwarz inequality and Lemma \ref{l:Abnd1d} for any $f\in \cM_2^{(1)}$ we can bound
$\ll Af,(\lambda-\SSEP)^{-1} Af\rr$ above by a constant multiple of 
\begin{align} \label{101}
 \ll A_{TASEP} f, (\lambda-\SSEP)^{-1} A_{TASEP}f\rr+ \ll \tilde A f, (\lambda-\SSEP)^{-1}\tilde Af\rr+ \ll f, (\lambda-\SSEP)f\rr 
\end{align}
 where $\tilde A$ is described in (\ref{A24})-(\ref{A35}). (Since we will work on $\cM_2^{(1)}$ we will only need (\ref{A24}).) Using the dimension reduced variables we may represent $f$ as a $\Z_+\to \R$ function $\bar f$  which is  constant on intervals of the form $[i\tK, (i+1)\tK-1]$ where $\tK=10K$. Then by (\ref{101}) and Lemma \ref{1dhardcore} we have
\begin{align}
\ll f, (\lambda-\SSEP) f\rr+\ll Af,(\lambda-\SSEP)^{-1} Af\rr&\le
C' \int_{-\pi}^\pi (\lambda+\lambda^{-1/2} t^2)|\hat F(t)|^2 dt
\end{align}
where $F(x)$ is the symmetric extension of $\bar f$. (We can assume that $\lambda$ is small enough.)

Since $f\in \cM_2$ we have
\begin{align}
\ll w, f\rr=\sum_{x\in\Z_+} \bar w_x \bar f(x), \qquad \textup{where}\quad w=\sum_\Lambda w_\Lambda \heta_\Lambda \quad \textup{and} \quad \bar w_x=\sum_{\Lambda=\{y,x+y\}} w_\Lambda.
\end{align}
Because of the local evolution condition on $r(\cdot,\cdot)$ and the definition of $w$ we have $\bar w_x=0$ for $x\ge \tK=10K$ and since $f\in \cM_2^{(1)}$ this means that 
\begin{align}
\ll w, f\rr=\bar f(0)\sum_{x\in\Z_+} \bar w_x =\bar f(0) \sum_{|\Lambda|=2}w_\Lambda=\frac{\chi}{2}j''(\rho)\bar f(0) \label{linearterm1}
\end{align}
where the last identity follows from (\ref{flux_der2}).

Now choose $F:\Z\to \R$ so that it's Fourier transform is given by\footnote{This is similar to the test functions used in \cite{LQSY}.}
\begin{align}
\hat F(t)=\hat G(2K t) \sum_{j=-\tK}^{\tK-1} e^{i j t},  \qquad \hat G(t)=\frac{1}{\lambda+\lambda^{-1/2} t^2} \ind(| t|\le 1/2)
\end{align}
It is not hard to see that $F(x)=F(-x-1)$ and that $F$ is piecewise constant on intervals of the form $[j\tK,(j+1)\tK-1]$. Moreover, 
\begin{align}
F(0)=\int_{-\pi}^{\pi} \hat F(t) dt&\ge C_{\tK} \lambda^{-1/4},\\
\int_{-\pi}^\pi (\lambda+\lambda^{-1/2} t^2)|\hat F(t)|^2 dt&\le C'_{\tK} \lambda^{-1/4}.
\end{align} 
We assumed $j''(\rho)>0$ in the first line, but we can do that without the loss of generality. (Otherwise use $-f$ for the test function.)
 Now if we choose our test function $f$ to be the restriction of $\alpha F(x)$ to $x\ge 0$ then by the previous estimates (and  (\ref{newvarpr})) we have
\begin{align}
\ll w, (\lambda-L)^{-1} w\rr\ge C_1 \alpha \lambda^{-1/4}-C_2 \alpha^2 \lambda^{-1/4}
\end{align}
and the desired lower bound follows by choosing $\alpha={C_1}/{2C_2}$. 
\end{proof}


\begin{proof}[Proof of the lower bound in  Theorem \ref{thm:1d} in $d=1$ when $j''(\rho)=0, j'''(\rho)\neq 0$]

By (\ref{flux_der}) we have  \\
$
\sum_{|\Lambda|=2} w_\Lambda=0$ but $\sum_{|\Lambda|=3} w_\Lambda\neq 0$. 
By (\ref{1dH-1}) and the Cauchy-Schwarz inequality, if we let $\tilde w=\sum_{|\Lambda|=3} c_\Lambda \heta_\Lambda$, we have 
$
\ll w, (\lambda-L)^{-1} w\rr\ge \frac12 \ll \tilde w, (\lambda-L)^{-1} \tilde w \rr-C$. 
We use the variational inequality (\ref{newvarpr}) with $\tilde w$ instead of $w$ and get the lower bound by choosing an appropriate test function $f\in \cM_3^{(1)}$. Using the dimension reduction (see (\ref{scproddred})) we may represent $f$ as a $\Z_+^2\to \R$ function.  By the Cauchy-Schwarz inequality and Lemma \ref{l:Abnd1d},  we can bound $\ll Af,(\lambda-\SSEP)^{-1} Af\rr$ by a constant multiple of
$
 \ll \tilde A f, (\lambda-\SSEP)^{-1}\tilde Af\rr+ \ll f, (\lambda-\SSEP)f\rr 
$, since the $A_{TASEP}$ part vanishes.
 By Lemma \ref{1dhardcore_infl} we  have the upper bounds
\begin{align}\notag
\ll f,(\lambda-\SSEP)f\rr+\ll Af,(\lambda-\SSEP)^{-1} Af\rr&\le C  \iint_{[-\pi,\pi]^2} (\lambda+t_1^2+t_2^2) |\hat F(t_1, t_2)|^2 dt_1 dt_2\\ \label{113}
&+ C \iint_{[-\pi,\pi]^2}   (t_1^2+t_2^2) |\log \lambda | |\hat F(t_1, t_2)|^2dt_1dt_2\\
&+C \int_{[-\pi,\pi]} \left| \int_{[-\pi,\pi]}\hat F(t_1, t_2) dt_2\right|^2 |t_1| dt_1.\notag
\end{align}
where $F$ is the symmetric extension of $f$ to $\Z_+^2$ described in (\ref{2dext}). Now choose $F$ so that its Fourier transform is given by 
\begin{align}
\hat F(t_1,t_2)&=\hat G(2\tK t_1,2\tK t_2) \sum_{k,j=-\tK-1}^{\tK} \exp(i k t_1+i j t_2),\\ 
\hat G(t_1,t_2)&=\ind({\lambda}\le t_1^2+t_2^2\le 1/2) \frac{-1}{(t_1^2+t_2^2)\log(t_1^2+t_2^2)}
\end{align}
with $\tK=10K$. 
Clearly  $F$ has the desired symmetry given by (\ref{2dext}), it is constant on the regions $[x\tK,(x+1)\tK-1]\times [y\tK,(y+1)\tK-1]$, $x,y\in \Z_+$ and also $F(x,y)=F(y,x)$. Thus if we choose $f$ to be the restriction of $F$ to $\Z_+^2$ then $f\in \cM_3^{(1)}$ and we have the upper bound (\ref{113}).  One can now explicitly bound the integrals on the right side of (\ref{113}) to get \begin{align}
\ll f,(\lambda-\SSEP)f\rr+\ll Af,(\lambda-\SSEP)^{-1} Af\rr&\le C_1 \log|\log \lambda|.
\end{align}
Using the same argument as the one leading to (\ref{linearterm1}) we have
\begin{align}
\ll \tilde w, f\rr=f(0,0)\sum_{|\Lambda|=3}w_\Lambda=\frac{\chi^{3/2}}{3!}j'''(\rho) f(0,0)=C  \iint_{[-\pi,\pi]^2} \hat F(t_1,t_2) dt_1dt_2\ge C_2  \log|\log \lambda|.
\end{align}
(We assumed $j'''(\rho)>0$, otherwise use $-f$ as the test function.)  Choosing $\alpha f$ as the test function in  (\ref{newvarpr}) with $\alpha={C_2}/{2C_1}$ completes the proof.  
\end{proof}


\begin{proof}[Proof of the lower bound in Theorem \ref{thm:1d}  in $d=2$ when $j_1''(\rho)\neq 0$]
We use the same strategy as before, here we will use a test function from $\cM_2^{(2)}$ in the inequality (\ref{newvarpr}) to get a lower bound. Suppose that $F:\Z^2\to \R$ satisfies
\begin{align}
F(x_1,x_2)=F(-x_1,-x_2), \qquad \textup{and} \qquad F(x_1,x_2)=F(0,0) \quad \textup{if} \quad  |x_1|,|x_2|\le \tK=10K. 
\end{align}
Define the test function $f\in \cM_2$ as
\begin{align}
f(\{(x_1,x_2), (y_1,y_2)\})=\begin{cases}
F(x_1,x_2)\qquad &\textup{if } \quad (y_1,y_2)=0,\\
0 \qquad &\textup{ otherwise}
\end{cases}
\end{align} 
it is easy to check that $f\in \cM_2^{(2)}$.
Then by Lemma \ref{l:A2d} and Lemma \ref{2dhardcore}  we have
\begin{align}
\ll f, (\lambda-\SSEP) f\rr+\ll A f,(\lambda-\SSEP)^{-1} A f\rr&\le C \iint_{[-\pi,\pi]^2} (\lambda+t_1^2+ t_2^2)| |\hat F(t_1,t_2)|^2 dt_1 dt_2\\
 &\hskip-50pt+ C \iint_{[-\pi,\pi]^2} (a t_1+b t_2)^2|\log \lambda | |\hat F(t_1,t_2)|^2 dt_1 dt_2\\
 &\hskip-50pt \le C'  \iint_{[-\pi,\pi]^2} (\lambda+(b t_1- a t_2)^2+ |\log \lambda| (a t_1+b_2 t_2)^2)| |\hat F(t_1,t_2)|^2 dt_1 dt_2\label{125}
\end{align}
Here $(a,b)=(j_1''(\rho), j_2''(\rho))$ and $C'$ does not depend on $\lambda$. If we choose
\begin{align}
\hat F(t_1,t_2)&=\hat G((2\tK+1) t_1,(2\tK+1) t_2) \sum_{k,j=-\tK}^{\tK} \exp(i k t_1+i j t_2),\\ 
\hat G(t_1,t_2)&=\ind( t_1^2+t_2^2\le 1/2) \left(\lambda+(b t_1- a t_2)^2+ |\log \lambda| (a t_1+b_2)^2\right)^{-1}.
\end{align}
then $F(x_1,x_2)=F(-x_1,-x_2)$ and $F(x_1,x_2)=F(0,0)$ if $|x_1|,|x_2|\le \tK$. We can explicitly estimate the integral on the right hand side of (\ref{125}) and show that it is bounded by $C_1\sqrt{|\log \lambda|}$. Using the same argument as the one leading  to (\ref{linearterm1})  we also get
\begin{align}
\ll w_1, f\rr=\frac{\chi}{2} j_1''(\rho) F(0,0)
=\frac{\chi}{2} j_1''(\rho) \iint_{[-\pi,\pi]^2} \hat F(t_1,t_2) dt_1dt_2\ge C_2 \sqrt{\log \lambda}.
\end{align}
The statement now follows by choosing $\alpha f$ as the test function with $\alpha={C_2}/{2C_1}$. Note that we assumed that $j_1''(\rho)>0$, but we may do that without the loss of generality.
\end{proof}

\section{Proof of the reduction lemma}\label{s:dual}


Using  (\ref{generator})  and $A=(L+L^*)/2$ 
we have 
\begin{align}
A f&=\tfrac12{\sqrt{\chi}} \sum_x \sum_{y} r(y,\tau_{-x} \eta)( \heta_{ x+y}-\heta_x) (f(\eta^{x,x+y})-f(\eta))
\end{align}
Let $f=\sum_\Lambda f_\Lambda \heta_\Lambda$ be local. We can also express $r(y,\eta)$ in our basis as $r(y,\eta)=\sum_{\Lambda} c_{y,\Lambda} \heta_\Lambda$, and because the rates are assumed to be local,  each $\Lambda$ is a subset of  $[-K,K]^d$ and  $\Lambda\cap \{0, y\}=\emptyset$.
Using the  identity 
\begin{align}
\heta_x^2=1+\rrr \heta_x, \qquad \rrr=(1-2\rho) \chi^{-1/2},\label{kappa}
\end{align}
we then have $
A=\sum_{y,\Lambda}A_{\Lambda,y}$ where 
\begin{align}
A_{\Lambda,y} f=\tfrac12{\sqrt{\chi}} c_{y,\Lambda} \sum_x  \sum_{ \Omega\cap\{x,x+y\}=\emptyset}  \heta_{\Lambda+x} (2+\rrr \heta_{x+y}+\rrr \heta_x-2 \heta_{\{x, x+y\}}) (f_{\Omega\cup \{x\}}- f_{\Omega\cup \{x+y\}}) \heta_{\Omega}\label{gensin}.
\end{align}
By (\ref{heta}) and  (\ref{kappa}) if $\Theta$ and $\Omega$ are disjoint then $\heta_{\Theta} \heta_\Omega=\heta_{\Theta\cup \Omega}$, otherwise we have $\heta_{\Theta} \heta_\Omega=\heta_{\Theta\triangle \Omega}+\rrr \heta_{\Theta \cup \Omega}$
where $X\triangle Y$ denotes the symmetric difference. Hence (\ref{gensin}) becomes \begin{align}
&
=\frac12{\sqrt{\chi}}  c_{y,\Lambda}\sum_x\Bigg\{  \sum_{ \substack{\Omega\cap (\Lambda+x)=\emptyset \\\Omega\cap\{x,x+y\}=\emptyset
} } (2 +\rrr \heta_{x+y}+\rrr \heta_x-2 \heta_{\{x, x+y\}}) (f_{\Omega\cup \{x\}}- f_{\Omega\cup \{x+y\}}) \heta_{\Omega\cup( \Lambda+x)}\label{A1stpiece}\\
&~+  \sum_{\emptyset\neq \Gamma \subset \Lambda }\sum_{ \substack{\Omega\cap (\Lambda+x)=\Gamma+x\\\Omega\cap\{x,x+y\}=\emptyset
} } (2+\rrr \heta_{x+y}+\rrr \heta_x-2 \heta_{\{x, x+y\}}) \label{A2ndpiece} (f_{\Omega\cup \{x\}}- f_{\Omega\cup \{x+y\}})(\heta_{\Omega\cup(\Lambda+x)\setminus(\Gamma+x)}+\kappa \heta_{\Omega\cup (\Lambda+x) })\Bigg\}.\end{align}
It becomes clear one needs some simplifying notation.
If $B_1$ and $B_2$ are disjoint  finite subsets of $\Z^d$ with  $0,y\in B_1\cup B_2$ and  
$B_3$ is a subset of $B_1\cup B_2$ with  $y\in B_3$, $0\notin B_3$
we let
\begin{equation}
A[B_1,B_2, B_3,y] f_\Omega=\sum_{x\in \Z^d} \indd{B_1+x\subset \Omega,B_2+x\cap \Omega=\emptyset} \left( f_{\left(\Omega\setminus(B_1+x)\right)\cup (B_3^{0,y}+x)} -f_{\left(\Omega\setminus(B_1+x)\right)\cup (B_3+x)}       \right)\label{bblock}
\end{equation}
where  $B+x$ means that we add $x$ to every element of $B$.
In plain words:\\
1.  For each $x\in \Z^d$ we check if the shift  of $\Omega$ matches the pattern given by $B_1$ and $B_2$ in the sense that  $\Omega$ contains $B_1+x$ and is disjoint with $B_2+x$.\\
2.  If it matches for a given $x$, then we replace the part $B_1+x$ in $\Omega$ by $B_3+x$ and $B_3^{0,y}+x$ respectively, evaluate $f$ at these two values and compute the difference ($B_3^{0,y}$ means that we delete $y$ and add the element $0$ to $B_3$)\\
3.   We sum over all the  $x$ where we have  a match.

If $\Omega=\{p_1,\dots,p_n\}$  in lexicographic order and the smallest element of $B_1$ is $q$. Then we can rewrite the previous line as 
\begin{align}
&\sum_{i=1}^{n} \indd{B_1+(p_i-q)\subset \Omega, B_2+(p_i-q)\cap \Omega=\emptyset}\left( f_{\left(\Omega\setminus(B_1+x)\right)\cup (B_3^{0,y}+(p_i-q))} -f_{\left(\Omega\setminus(B_1+(p_i-q))\right)\cup (B_3+(p_i-q))}       \right)\label{bblock1}
\end{align}
We can also compute the dimension reduced form of this. Let $\Lambda\in\mathcal B_n[d]$ and assume that $\Omega=\{p_1,\dots,p_n\}$ is in the  equivalence class $\Lambda$.  A simple computation gives
\begin{align}
\overline{A[B_1,B_2, B_3,y] f}_\Lambda&=\sum_{i=1}^{n} \indd{B_1+(p_i-q)\subset \Omega\textup{ and } B_2+(p_i-q)\cap \Omega=\emptyset}(\bar f_{\tilde \Lambda_i}-\bar f_{\Lambda_i}).\label{bblock_dred}
\end{align}
Here $\bar f$ is the dimension reduced form of $f$ (see (\ref{scproddred})), and $\tilde \Lambda_i, \Lambda_i$ are the equivalence classes of the sets $\left(\Omega\setminus(B_1+x)\right)\cup (B_3^{0,y}+(p_i-q))$ and $\left(\Omega\setminus(B_1+(p_i-q))\right)\cup (B_3+(p_i-q))$, respectively.
Note that this can now be written as the sum of operators of the form $A[B_1,B_2,B_3,y]$. Indeed, the sum (\ref{A1stpiece}) is just $\tfrac12{\sqrt{\chi}}  c_{y,\Lambda}$ times \begin{align}\label{A11}
2 A[\Lambda,\{0,y\},\{y\},y]+\rrr A[\Lambda\cup\{y\},\{0\},\{y\},y]+\rrr A[\Lambda\cup\{0\},
\{y\},\{y\},y]
-2 A[\Lambda\cup\{0,y\},\emptyset,\{y\},y].
\end{align}
The sum (\ref{A2ndpiece}) will have similar form, for each $\emptyset\neq \Gamma \subset \Lambda$ we get $\tfrac12{\sqrt{\chi}}  c_{y,\Lambda}$ times the following sum:
\begin{align}
&2 \rrr A[\Lambda,\{0,y\},\Gamma \cup\{y\},y]+\rrr^2 A[\Lambda\cup\{y\},\{0\},\Gamma \cup\{y\},y]\\ \nonumber&
\hskip100pt+\rrr^2 A[\Lambda\cup\{0\},
\{y\},\Gamma \cup\{y\},y]-2\rrr A[\Lambda\cup\{0,y\},\emptyset,\Gamma\cup\{y\},y]\\\label{A22}
&2  A[\Lambda\setminus \Gamma,\Gamma\cup \{0,y\},\Gamma \cup \{y\},y]+\rrr A[\Lambda\cup\{y\}\setminus \Gamma,\Gamma \cup \{0\},\Gamma \cup \{y\},y]\\&
\hskip100pt+\rrr A[\Lambda\cup\{0\}\setminus \Gamma,
\Gamma \cup \{y\},\Gamma \cup \{y\},y]-2 A[\Lambda\cup\{0,y\}\setminus \Gamma,\Gamma,\Gamma \cup \{y\},y].\nonumber
\end{align}
We will show that many of the operators $A[B_1,B_2,B_3,y]$ can be bounded using the Dirichlet form, and that in a certain sense the important quantities about such an operator are $|B_1|$ and $|B_3|$. The following lemma summarizes what we have proved so far. 
\begin{lemma} \label{l:bookkeeping}Consider a fixed finite $\Lambda$ with $|\Lambda|=\ell$ and $y\in \Z^d$. Then the operator $A_{\Lambda,y}$ of (\ref{gensin}) can be expressed as a finite linear combination of operators of the form $A_{\Lambda,y}=\sum \alpha_{B_1,B_2,B_3} A[B_1,B_2,B_3,y]$. The sum of the coefficients $ \alpha_{B_1,B_2,B_3}$ with  $|B_1|, |B_2|, |B_3|$  fixed is given by the following table:
\begin{align}
\begin{array}{|c|c|c|c|}\hline
|B_1|&|B_2|&|B_3|&\textup{sum of coefficients}\\
\hline
\ell & 2&1& \sqrt{\chi} c_{y,\Lambda} \\
\hline
\ell+1 & 1&1& \rrr \sqrt{\chi} c_{y,\Lambda}  \\
\hline
\ell+2 & 0 &1& -\sqrt{\chi} c_{y,\Lambda}\\
\hline
\ell & 2&a+1& \rrr \sqrt{\chi} c_{y,\Lambda} \binom{\ell}{ a} \\
\hline
\ell+1 & 1&a+1& \rrr^2 \sqrt{\chi} c_{y,\Lambda}  \binom{\ell}{ a} \\
\hline
\ell+2 & 0 &a+1& -\rrr \sqrt{\chi} c_{y,\Lambda} \binom{\ell}{ a} \\
\hline
\ell-a & a+2&a+1& \rrr \sqrt{\chi} c_{y,\Lambda} \binom{\ell}{ a} \\
\hline
\ell+1-a & a+1&a+1& \rrr^2 \sqrt{\chi} c_{y,\Lambda} \binom{\ell}{ a}\\
\hline
\ell+2 -a& a &a+1& -\rrr \sqrt{\chi} c_{y,\Lambda} \binom{\ell}{ a}\\
\hline
\end{array}\label{ratetable}
\end{align}
where $1\le a \le \ell$.

\end{lemma}


{\it Examples.} 
\textit{1. Exclusion process. } We consider the 1d TASEP first, in that case the only non-zero $c_{y, \Lambda}$ is $c_{1,\emptyset}=1$. Then $A_+, A_-,  A_0$ can be expressed as
\begin{align}\label{A0TASEP_block}
A_0&=\frac{\kappa}{2} \sqrt{\chi}\left(A[\{1\},\{0\},\{1\},1]+  A[\{0\},\{1\},\{1\},1]     \right)\\
A_+&=-\sqrt{\chi} A[\{0,1\},\emptyset, \{1\},1], \qquad A_-=\sqrt{\chi} A[\emptyset, \{0,1\},\{1\},1].\label{ATASEP_block}
\end{align}
In case of a general exclusion with jump law $p(\cdot)$ the operator $A$ is just a linear combination of terms like this with $y$ in place of $1$ if $p(y)>0$.

\textit{2. The model from (\ref{simplerates}). }
For the model (\ref{simplerates}) at density $\rho=1/2$,   $\rrr=0$, so many terms vanish. The rates  in our basis are
$
r(1,\eta)=2-\frac12 \heta_{-1}-\frac12 \heta_{2}$, $r(-1,\eta)=2$,
and only the $-\frac12 \heta_{-1}-\frac12 \heta_{2}$ part contributes to the asymmetric part. 
This means that 
\begin{align}
4 A=&A[\{-1\},\{0,1\},\{1\},1]-
A[\{-1,0,1\},\emptyset,\{1\},1]+
A[\{2\},\{0,1\},\{1\},1]-A[\{0,1,2\},\emptyset,\{1\},1]\nonumber\\
&+ A[\emptyset,\{-1,0,1\},\{-1,1\},1]- A[\{0,1\},\{-1\},\{-1,1\},1]\\\nonumber &+ A[\emptyset,\{0,1,2\},\{1,2\},1]- A[\{0,1\},\{2\},\{1,2\},1]
\end{align} 
%


\begin{lemma}\label{l:Abound1} Suppose $B_1,B_2,B_3,y$ and $B_1',B_2',B_3,y$ are as in the preamble to
(\ref{bblock}) and
 $B_1\cup B_2=B_1'\cup B_2'=[-\ell,\ell]^d$. If $|B_1|=|B_1'|$ then  $A[B_1,B_2,B_3,y]-A[B_1',B_2',B_3,y]$ is  sectorial.
\end{lemma}
\begin{proof} We will prove sectoriality first  with respect to the scalar product $\l\cdot, \cdot\r$. 

We have to bound $\la f, (A[B_1,B_2,B_3,y]-A[B_1',B_2',B_3,y]) g\ra$ by Dirichlet forms.
Using the definition of the operators and rearranging the corresponding sums, it can be written as
\begin{align}\label{AB1B2B3}
\sum_{x} \sum_{\Gamma \cap ([-\ell,\ell]^d+x)=\emptyset} (f_{\Gamma\cup(B_1+x)}-f_{\Gamma\cup(B_1'+x)})(g_{\Gamma\cup (B_3^{0,y}+x)}-g_{\Gamma\cup(B_3+x)})
\end{align}
Since $|B_1|=|B_1'|$  we can get from $B_1$ to $B_1'$ by finitely many nearest neighbor steps (all of which take place inside $[-\ell,\ell]^d$): $\Omega^0_1=B_1, \Omega_1^1,\dots, \Omega_1^{m_1}=B_1'$ where we get $\Omega_1^{i+1}$ from $\Omega_1^i$ by moving one of its elements by a unit vector. 
Thus we may rewrite $f_{\Gamma\cup(B_1+x)}-f_{\Gamma\cup(B_1'+x)}$ as a sum of discrete gradients:
\begin{equation}
f_{\Gamma\cup B_1}-f_{\Gamma\cup B_1'}=\sum_{i=0}^{m_1-1} f_{\Gamma\cup (\Omega_1^i+x)}-f_{\Gamma\cup (\Omega_1^{i+1}+x)}.
\end{equation}
Since $|B_3|=|B_3^{0,y}|$ we have an analogous  representation for $g_{\Gamma\cup (B_3^{0,y}+x)}-g_{\Gamma\cup(B_3+x)}$ with a series of sets    $\Omega^0_2=B_3, \Omega_2^1,\dots, \Omega_2^{m_2}=B_3^{0,y}$
By Cauchy-Schwartz inequality the left hand side of (\ref{AB1B2B3}) is thus bounded above by 
$\sqrt{ \tilde{D}_1(f)\tilde{D}_2(g) }$ where 
\begin{align}\label{tildeDf}
\tilde{D}_1(f)= \sum_{x} \sum_{\Gamma \cap [-\ell,\ell]^d+x=\emptyset} \sum_{i=0}^{m_1-1} (f_{\Gamma\cup (\Omega_1^i+x)}-f_{\Gamma\cup (\Omega_1^{i+1}+x)})^2\end{align}
and $\tilde D_2(g)$ is defined similarly. 
The  right hand side of (\ref{tildeDf}) can be bounded by constant multiple of $\la f, (-\SSEP) f \ra$ (with a constant depending only on $\ell$) since each term is of the form $(f_\Lambda-f_\Lambda')^2$ (where $\Lambda'$ is obtained via a nearest neighbor step from $\Lambda$) can only appear finitely many times in that sum. This is because the bond where the jump happened between $\Lambda$ and $\Lambda'$ must be inside the cube $[-\ell,\ell]^d+x$. The same bound holds for $\tilde D_2(g)$ with $\la g, (-S) g\ra$ which yields
\begin{equation}
\la f, (A[B_1,B_2,B_3,y]-A[B_1',B_2',B_3,y]) g\ra^2\le C_\ell \la f, (-S) f\ra \la g, (-S) g\ra.
\end{equation}
This shows the sectoriality for the scalar product $\l\cdot, \cdot\r$. The statement of the lemma will follow from (\ref{singledouble}) which shows how to get the appropriate bound for $\ll\cdot,\cdot\rr$ from $\l\cdot, \cdot\r$.
\end{proof}

\begin{lemma}\label{l:Abnd3}
Assume that $B_1\cup B_2=[-\ell,\ell]^d$, and $B_1, B_2, B_3,y$ are as in the preamble to
(\ref{bblock}) . Then if $d=1$ and $|B_1|\ge 4$, or if $d=2$ and $|B_3|\ge 3$, then for $T=A[B_1,B_2,B_3,z]$ and $f\in \cM_n$ we have 
\begin{align}
\ll T f, (\lambda-S)^{-1} T f \rr\le C n \ll f, (-S) f\rr
\end{align}
where $C$ only depends on $\ell$. 
\end{lemma}
\begin{proof}
We will use the notation $b_i=|B_i|$. From the definition (\ref{bblock}) it follows that if $f\in \cM_n, n\ge b_3$ then $T f\in \cM_{\tilde n}$ with $\tilde n=n-b_3+b_1$. We can clearly assume that $n\ge b_3, \tilde n\ge b_1$.

We will start with the proof of the \underline{one dimensional case}. By Lemma \ref{l:Abound1} we may assume that $B_1=[-\ell,-\ell+b_1-2]\cup\{\ell\}$. 
In the dimension reduced form our operator $T=A[B_1,B_2,B_3,y]$ is
\begin{equation}\label{Tdred}
Tf(y_1,\dots,y_{\tilde n-1})=\sum_{i=1}^{\tilde n-b_1+1} \ind(y_i=y_{i+1}=\dots=y_{i+b_1-3}=0, y_{i+b_1-2}= 2\ell-k) (f( \tau_i' \und y)-f(\tau_i \und y )).
\end{equation}
Here $\tau_i \und y$ and $\tau_i' \und y$ are defined the following way. Let $\Omega\subset \Z$ be the size $\tilde n$ set whose dimension reduced form is $\und y=(y_1,\dots, y_{\tilde n-1})$ and its $i^{th}$ largest element is $-\ell$ (this will uniquely determine $\Omega$). Then $\tau_i \und y$ and $\tau_i' \und y$ are defined as the dimension reduced forms of 
$
(\Omega\setminus B_1)\cup B_3$ and  $(\Omega\setminus B_1)\cup B_3^{0,z}$.
We will use the shortened notation $\ind_i$ for the indicator function in (\ref{Tdred}).
By the variational formula 
\begin{align}\notag
\ll Tf (-\SSEP)^{-1} Tf \rr
&= \sup_g \sum_{i=1}^{\tilde n-b_1+1} 2 \ll g ,\ind_i (f\circ \tau_i'-f\circ \tau_i)\rr -\ll g, (-\SSEP) g \rr\label{varT}
\end{align}
Consider the symmetric simple exclusion generator $S_{b_1}$ on $\cM_{b_1}$ in the dimension reduced form, i.e.~acting on functions of the form $g:\Z_+^{b_1-1}\to \R$ (see (\ref{symgen})). Now let $S_{i, b_1}$ be the operator $S_{b_1}$ acting on dimension reduced variables $y_i, y_{i+1}, \dots, y_{i+b_1-2}$ of a $\cM_{\tilde n}$ function.  This is the generator of symmetric simple exclusion performed by the particles $i, i+1, \dots, i+b_1-1$ where particle $i$ is `glued' to the particles $1, 2, \dots, i-1$ and $i+b_1-1$ is glued to the particles $i+b_1, \dots, \tilde n$: i.e. whenever $i$ or $i+b_1-1$ jumps, the corresponding particles perform the same jump simultaneously. We claim that for any $g\in \cM_n$ we have
\begin{equation}
\ll g, -\left(\sum_{i=1}^{\tilde n-b_1+1} S_{i,b_1}\right) g\rr\le 2 \tilde n \ll g, (-\SSEP) g \rr\label{S155}
\end{equation}
The operator $\sum_{i=1}^{\tilde n-b_1+1} S_{i,b_1}$ corresponds to the dimension reduced form of a random walk where we can move each particle left or right if possible with rate at most $b_1$ or we can move the leftmost $k$ or rightmost $k$ ($1\le k \le \tilde n-b_1+1$) particles together one unit left or right with rate one. Note that whenever such a simultaneous jump happens we can recreate this jump by taking at most $\tilde n-b_1+1$ nearest neighbor jumps with the particles. From this the upper bound (\ref{S155}) follows by the Cauchy-Schwarz inequality.

By Lemma \ref{l:harmonic} we have 
\begin{equation}
\ll g, - S_{j,b_1} g\rr\ge C \ll \ind_j g, \ind_j g \rr. 
\end{equation}
Here we first fix the variables $y_1, \dots, y_{j-1}, y_{j+b_1-1}, y_{j+b_1}, \dots, y_{\tilde n-1}$, apply the lemma and then average out in the variables we fixed.

Using this with the bound (\ref{S155}) and the usual  variational formula for $\ll g,(-\SSEP)^{-1} g\rr$ we get
\begin{align}
\ll Tf, (-\SSEP)^{-1} Tf \rr &\le \sup_g \sum_{j=1}^{\tilde n-b_1+1} \left(2 \ll g, \ind_j (f\circ \tau_j'-f\circ \tau_j) \rr -\frac1{2\tilde n}\ll g, - S_{j,k} g\rr\right)\\
&\le \sup_g \sum_{j=1}^{\tilde n-b_1+1} 2 \ll \ind_j g, \ind_j (f\circ \tau_j'-f\circ \tau_j)\rr - \frac1{2 C \tilde n}\ll \ind_j g,\ind_j  g\rr\\
&\le 2C  \tilde n \sum_{j=1}^{\tilde n-b_1+1}  \ll \ind_j(f\circ \tau_j'-f\circ \tau_j), \ind_j (f\circ \tau_j'-f\circ \tau_j) \rr\\
&\le C' \tilde n \ll f, (-\SSEP) f\rr.
\end{align}
Here the last step is just Cauchy-Schwarz again: for each $j$ the difference $f\circ \tau_j'-f\circ \tau_j$ describes a size $z$ jump of a particle which can be recreated by $z$ nearest neighbor jumps. Each such jump can appear for at most $b_1$ of the indices $j$ from which the last inequality follows with a $C_1$ depending on $b_1, z$.  Since  $\tilde n \le c n$ with a suitable $c$ depending only on $b_1, b_3$ the statement of the lemma follows.

Now we will turn to the \underline{two dimensional case}.  Using Lemma \ref{l:Abound1} we may assume that $(-\ell, -\ell)$ and $(\ell, \ell)$ are elements of $B_1$. We will work with the dimension reduced picture, recall that $\mathcal{B}_n[2]$ is the set of equivalence classes of the size $n$ subsets of $\Z^2$ under shifts ($\Lambda\sim \Lambda+x, x\in \Z^2$).

Let $p_i$ denote the $i^{th}$ element of $\Lambda\in \mathcal B_{\tilde n}[2]$ in the lexicographic order (this is well defined as the order does not change when we shift the set). Then we have
\begin{align}\notag
\ll Tf (-\SSEP)^{-1} Tf \rr&=\sup_g 2 \ll T f, g\rr-\ll g, (-\SSEP) g \rr\\
&= \sup_g \sum_{i=1}^{\tilde n-b_1+1} 2 \ll g ,\ind_i (f\circ \tau_i'-f\circ \tau_i)\rr -\ll g, (-\SSEP) g \rr\label{varT2}
\end{align}
Here $\ind_i(\Lambda)$ is the indicator of the event that if we consider the shifted version of the box $[-\ell, \ell]^2$ whose lower right corner is exactly $p_i$ then the intersection with $\Lambda$ is exactly the appropriate shifted version of $B_1$. (This is basically the same as the $\ind_i$ in the one-dimensional case.) The operators $\tau_i, \tau_i'$ are just the analogues of their one dimensional counterparts: we replace the intersection of the box and $\Lambda$ with the shifted versions of $B_3$ and $B_3^z$. 

The rest of the proof is  similar to the one dimensional case: we will eventually prove that for any $g\in \cM_{\tilde n}$ we have
\begin{align}
C \tilde n \ll g, (-\SSEP) g\rr\ge \sum_{i=1}^{\tilde n-b_1+1} \ll \ind_i g, \ind_i g\rr. \label{bound163}
\end{align}
From this the statement will follow exactly the same way using (\ref{varT2}) and the end of the one dimensional argument.

Fix $1\le i \le \tilde n-b_1+1$ and suppose that for $\Lambda\in \mathcal B_{\tilde n}[2]$ we have $\ind_i (\Lambda)=1$. We can move the $i-1$ particles which are in front of $p_i$ in the lexicographic order so that their $x$-coordinates are  smaller than that of $p_i$. Clearly we can do this by moving each of these particles one step to the left. We can also move the particles which are not in the square $p_i+[0,2\ell]^2$ and not in front of $p_i$ so that their $x$-coordinates are bigger than $2\ell$ plus the $x$-coordinate of $p_i$. This can be achieved by moving each such particle $2\ell+1$ steps to the right. Denote the new configuration by $\sigma_i( \Lambda)$. (If $\ind_i (\Lambda)=0$ then we may define $\sigma_i (\Lambda)$ as the empty set.)  
\begin{figure}[h!]\label{movpart}
  \centering
\includegraphics[height=150pt]{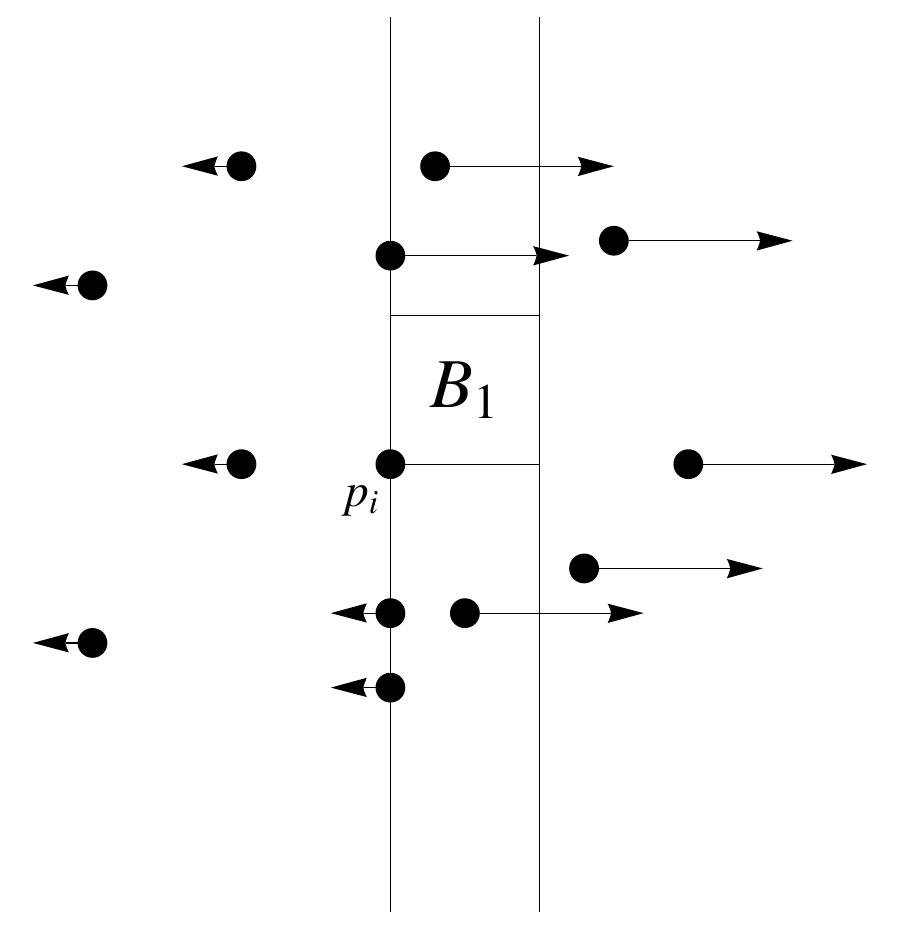}\qquad \includegraphics[height=150pt]{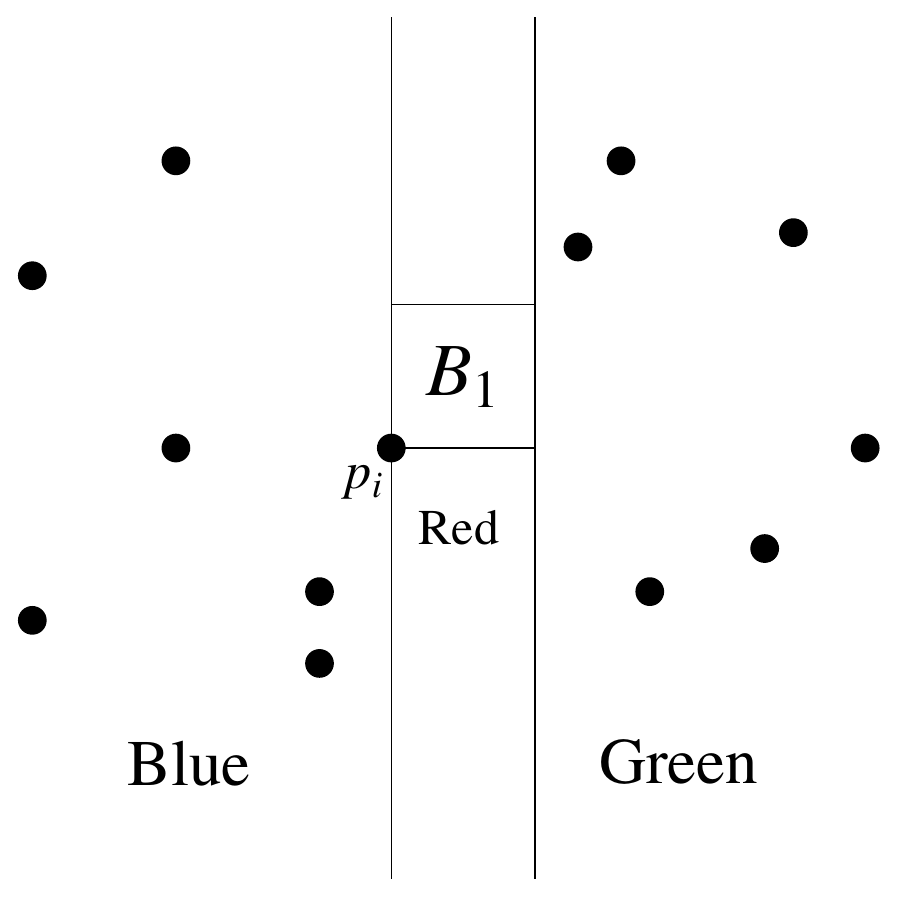}
 \caption{Schematic representation of $\sigma_i$}
\end{figure}

Note that because we made at most $2\ell+1$ horizontal steps with each particle to go from $\Lambda$ to $\sigma_i( \Lambda)$ we have 
\begin{align}
\ll \ind_i(g-g\circ\sigma_i), \ind_i(g-g\circ\sigma_i) \rr\le C(\ell) \ll g, (-\SSEP) g \rr. 
\end{align}
We will now color the particles of $\sigma_i \Lambda$ with red, blue and green so that the particles that are in $p_i+[0,2\ell]^2$ initially are red, the particles which have smaller $x$-coordinates than $p_i$ are blue and the rest are green.  
 Now define $S_{i, b_1}$ similarly as before: it will be a symmetric simple exclusion on the $b_1$ red particles, but whenever we would change the smallest red $x$-coordinate then we perform the same change on all the blue points simultaneously (i.e. move them one step to the left or right) and whenever we would change the largest red $x$-coordinate then we will perform a similar step on all the green particles. It is clear from Lemma \ref{l:harmonic}  that 
 \begin{align}
 \ll \ind_i \,g\circ\sigma_i, \ind_i \,g\circ\sigma_i \rr\le C \ll g, (-S_{i, b_1}) g \rr 
 \end{align}
 We also have (by the same argument as in (\ref{S155})) 
 \begin{equation}
\ll g, -\left(\sum_{i=1}^{\tilde n-b_1+1} S_{i,b_1}\right) g\rr\le 2 \tilde n \ll g, (-\SSEP) g \rr\label{S155new}.
\end{equation}
 Putting together our bounds  will yield (\ref{bound163}) (with an $\ell$ dependent constant $C$) and this completes the proof of the lemma.
\end{proof}
\begin{remark}\label{r:Abnd}
If $0\neq z\notin B_1\cup B_2$ then 
\begin{align}\label{l:Aexpand}
A[B_1,B_2,B_3,y]=A[B_1,B_2\cup\{z\},B_3,y]+A[B_1\cup\{z\},B_2,B_3\cup\{z\},y].
\end{align}
This follows immediately from the definition (\ref{bblock}), one just has to rewrite the indicator on the right hand side as
$
\indd{(B_1\cup\{z\})+x\subset \Omega\textup{ and } B_2+x\cap \Omega=\emptyset}+\indd{B_1+x\subset \Omega\textup{ and } (B_2\cup\{z\})+x\cap \Omega=\emptyset}.
$

Note that this means that we  do not have to assume that $B_1\cup B_2=[-\ell,\ell]^d$ Lemma \ref{l:Abnd3}. 
Indeed, suppose that $d=1$, $B_1\cup B_2\subset [-\ell,\ell]$ and  $|B_1|>3$. Then by repeatedly applying (\ref{l:Aexpand}) we may rewrite $A[B_1,B_2,B_3,y]$ as a finite linear combination of  $A[B_1',B_2',B_3',y]$ where $B_1\subset B_1'$ and $B_1'\cup B_2'=[-\ell,\ell]$. One can now apply Lemma \ref{l:Abnd3} for each term to get the sectorial bound for  $A[B_1,B_2,B_3,y]$. A similar proof applies in the $d=2$ case.
\end{remark}

\begin{lemma}\label{l:A0bnd}
Let $T=A[\{z\},\{0\},\{z\},z]$ with some $0\neq z\in \Z^d$. Then if $f\in \cM_n$ then
\begin{align}
\ll T f,(\lambda-\SSEP)^{-1} Tf \rr\le C n^2 \ll f, (\lambda-\SSEP) f \rr. 
\end{align} 
\end{lemma}
\begin{remark}
Using the representation (\ref{A0TASEP_block}) with Lemmas \ref{l:Abound1} and \ref{l:A0bnd} we immediately get that the TASEP $A_0$ is sectorial on $\cM_n$. \end{remark}
\begin{proof}[Proof of Lemma \ref{l:A0bnd}] We first deal with the case $d=1, z=1$, the general case will be similar. 
Note that on $\cM_n$ the operator $T$ is exactly the generator of an exclusion process (with $n$ particles) with rate one unit jumps to the right. 
We can also consider this as the generator of a random walk on the elements of $\mathcal B_n[1]$. For an element  $\Lambda\in \mathcal B_n[1]$   denote by $\tau_k \Lambda$ the effect of moving the $k^{th}$ largest particle one step to the right if possible. In the dimension reduced form this can be written as 
\begin{equation}
\tau_k(y_1,\dots, y_{n-1})=\begin{cases}
(y_1, \dots, y_{i-1}+1, y_i-1,y_{i+1} \dots) \quad &\textup{if } y_i>0\\
(y_1,\dots, y_{n-1})\qquad &\textup{otherwise}\\
\end{cases}
\end{equation}
We denote by $\sigma_k$ the effect of moving the first $k$ particles one step to the left (with $\sigma_0$ being the identity). In the dimension reduced form this is just changing $y_k$ to $y_k+1$:
\begin{equation}
\sigma_k(y_1,\dots, y_{n-1})=(y_1, y_2, \dots, y_{k-1}, y_k+1, y_{k+1},\dots).
\end{equation}
Note that $\tau_\ell \sigma_{\ell}=\sigma_{\ell-1}$ and in the dimension reduced form $\sigma_n$ is the identity. (Note that $\tau$ and $\sigma$ are different from the operators used in the proof of the previous lemma.)
By definition we have
\begin{equation}
Tf =\sum_{i=1}^n (f\circ \tau_i -f)
\end{equation}
The last term in the sum can be rewritten as 
\begin{equation}
f\circ \tau_n -f=f \circ \sigma_{n-1}-f=\sum_{i=1}^{n-1} (f\circ  \sigma_{n-i} -f \circ \sigma_{n-i-1})=\sum_{i=1}^{n-1} (f\circ  \sigma_{n-i} -f\circ \tau_{n-i} \circ \sigma_{n-i})
\end{equation}
which gives
\begin{equation}
T f =\sum_{i=1}^{n-1} (f\circ \tau_i -f)-\sum_{i=1}^{n-1} (f\circ \tau_i \circ \sigma_{i} -f\circ  \sigma_{i})
\end{equation}
Computing $\ll g, T f \rr$ we get
\begin{eqnarray*}
\ll g, T f \rr&=& \sum_{\und y\in \Z_{+}^{n-1}} \sum_{i=1}^{n-1} (f(\tau_i \und y)-f(\und y)) g(\und y)-
\sum_{\und y\in \Z_+^{n-1}} \sum_{i=1}^{n-1} (f(\tau_i \sigma_{i} \und y)-f( \sigma_{i}\und y)) g(\und y)\\
&=& \sum_{\und y} \sum_{i=1}^{n-1} (f(\tau_i \und y)-f(\und y)) g(\und y)-\sum_{\und y} \sum_{i=1}^{n-1} (f(\tau_i  \und y)-f( \und y)) g(\sigma_i^{-1} \und y) \ind(y_i>0)\\
&=& \sum_{\und y} \sum_{i=1}^{n-1} (f(\tau_i \und y)-f(\und y)) (g(\und y)-g(\sigma_i^{-1} \und y)) \ind(y_i>0)
\end{eqnarray*}
since $(f(\tau_i \und y)-f(\und y))\ind(y_i=0)=0$ by definition. Note that if $y_i>0$ then 
\begin{align}
\sigma_i^{-1} \und y=(y_1,\dots ,y_{i-1},y_i-1, y_{i+1},\dots)=\tau_1\tau_2\cdots \tau_i \und y
\end{align}
 and
\begin{equation}
(g(\und y)-g(\sigma_i^{-1} \und y)) \ind(y_i>0)=- \ind(y_i=0) \sum_{j=1}^i (g(\tau_{j} \tau_{j+1}\cdots \tau_i \und y)-g(\tau_{j+1} \tau_{j+2}\cdots \tau_i \und y))
\end{equation}
which means that 
\begin{equation}
\ll g, T f \rr=\sum_{\und y} \sum_{i=1}^{n-1} (f(\tau_i \und y)-f(\und y)) \sum_{j=1}^i (g(\tau_j \und y_{i,j})-g(\und y_j))
\end{equation}
where $\und y_{i,j}=\tau_{j+1} \tau_{j+2}\cdots \tau_i \und y$. Applying the Cauchy-Schwarz inequality and noting that each $f$ gradient is multiplied by the sum of $i<n$ $g$ gradients we get
\begin{align}
\ll g, T f \rr^2&\le  \sum_{\und y} \sum_{i=1}^{n-1} (f(\tau_i \und y)-f(\und y))^2 \times \sum_{\und y} \sum_{i=1}^{n-1}i  \sum_{j=1}^i (g(\tau_j \und y_{i,j})-g(\und y_{i,j}))^2\\
&\le n^2 \sum_{\und y} \sum_{i=1}^{n-1} (f(\tau_i \und y)-f(\und y))^2 \times\sum_{\und y} \sum_{i=1}^{n-1} (g(\tau_i \und y)-g(\und y))^2 \\
&=4n^2 \ll f, (-\SSEP) f\rr \ll g, (-\SSEP) g\rr
\end{align}
from which the lemma follows for $d=1, z=1$ by taking $g=(\lambda-\SSEP)^{-1} T f$. 

If $d=1$ and $z>1$ then basically the same proof works. In this case the exclusion with jumps $z$ on $\Z$ can be decoupled into exclusions on the sub-lattices  $z\Z, z\Z+1, \dots z\Z+z-1$ and on each of those sub-lattices we can apply the $z=1$ result. Similar trick works for $d>1$: we can decouple the exclusion on $\Z^d$ with fixed jumps of size $z$ into exclusions on one-dimensional sub-lattices of the form $a+z \Z$ and we can again use the $d=1, z=1$ result.
\end{proof}

\begin{lemma}\label{l:Asimple}
1.  Assume that $B_1,B_2,B_3, y$ are as in the preamble to
(\ref{bblock})  with $d=1$ and we have $B_1,B_2,B_3\subset[-K,K], |y|\le K$.  If $|B_3|>1$ then  
 $A[B_1,B_2,B_3,y]$ will vanish on $\cM_2^{(1)}$ after the dimension reduction, i.e.~for any $f\in \cM_2^{(1)}$ and $g$ we have
$
 \ll g, A[B_1,B_2,B_3,y]f\rr=0$. Moreover, if $|B_3|>2$ then  $A[B_1,B_2,B_3,y]$ will vanish on $\cM_3^{(1)}$ after the dimension reduction.

2.  Assume that $B_1,B_2,B_3, y$ are as in the preamble to
(\ref{bblock})  with $d=2$ and we have $B_1,B_2,B_3\subset[-K,K]^2, |y|\le K$.  If $|B_3|>1$ then   
 $A[B_1,B_2,B_3,y]$ will vanish on $\cM_2^{(2)}$ after the dimension reduction. 
\end{lemma}
\begin{proof}
From definition (\ref{bblock}) it is clear that $A[B_1,B_2,B_3,y]$ vanishes on $\cM_n$ if $n<|B_3|$. In order to prove the statements of the lemma we just have to understand $A[B_1,B_2,B_3,y]$ on $\cM_{|B_3|}$ after the dimension reduction. If $f\in \cM_{|B_3|}$ then $\tilde \Lambda_i$ and $\Lambda_i$ in (\ref{bblock_dred}) are always equal to the equivalence classes of $B_3^{0,y}$ and $B_3$, respectively. Thus if $\bar f$ always gives the same value on these equivalence classes then $\overline{A[B_1,B_2,B_3,y]f}$ is always zero which will make the scalar product $
 \ll g, A[B_1,B_2,B_3,y]f\rr$ vanish as well by (\ref{scproddred}). From the definitions (\ref{M21})-(\ref{M22}) of $\cM_2^{(1)}, \cM_3^{(1)}$ and $\cM_2^{(2)}$ the lemma now follows.
\end{proof}

%
%
%

\begin{proof}[Proof of Lemma \ref{l:Abnd1d}]
Write the asymmetric part of the generator $A$  as a finite linear combination 
$
A=\sum c_{B_1,B_2,B_3,y} A[B_1,B_2,B_3,y]$ with $B_i\subset [-K,K]$ and $|y|\le K$.

By Lemma \ref{l:Abnd3} and Remark \ref{r:Abnd}  all the terms with $|B_1|>3$ are sectorially bounded. By Lemma \ref{l:Asimple} we do not have to worry about the terms with $|B_3|>1$ on $\cM_2^{(1)}$ and the terms with $|B_3|>2$ on $\cM_3^{(1)}$ because they will vanish after the dimension reduction. We will now treat the various remaining terms according to the size $|B_3|-|B_1|$.

\noindent \textbf{Case 1: $\mathbf{|B_1|-|B_3|=2.}$} Since $|B_1|\le 3$ the only possibility here is $|B_1|=3, |B_3|=1$. We will show that these terms can be written as a constant multiple of $A[\{-1,0,1\},\emptyset,\{1\},1]$ and a term which is sectorial on $\cM_2\cup\cM_3$. As a first step we note that we may assume that $B_1\cup B_2=[-\ell,\ell]$ for some $\ell$, one just needs to apply the argument in Remark \ref{r:Abnd}  and use the fact that the building blocks $A[B_1',B_2',B_3',y]$ with $|B_1'|>3$ are all bounded. 
Next we can assume that $y=1$: in the definition (\ref{bblock}) we can just replace the gradient
\begin{align}
 f_{\left(\Omega\setminus(B_1+x)\right)\cup (B_3^{0,y}+x)} -f_{\left(\Omega\setminus(B_1+x)\right)\cup (B_3+x)}     
\end{align}
by a sum of gradients involving nearest neighbor switches $B_3\to B_3^{z,z-1}$. Then by Lemma \ref{l:Abound1} we may assume that $B_1=\{-1,0,1\}$ so we only need to deal with $A[\{-1,0,1\},B_2,\{1\},1]$. Finally, by using (\ref{l:Aexpand}) repeatedly we can write $A[\{-1,0,1\},B_2,\{1\},1]$ as a sum of $A[\{-1,0,1\},\emptyset,\{1\},1]$ and 
finitely many building blocks of the form $A[\{-1,0,1\}\cup \Gamma,B_2\setminus \Gamma,\{1\},1]$ with $|\Gamma|>0$ which are all sectorial by Lemma \ref{l:Abnd3}.

From the definitions it is easy to check that $A[\{-1,0,1\},\emptyset,\{1\},1]$ acting on $\mathcal M_2$ and $\mathcal M_3$ can be written as (\ref{A24}) and (\ref{A35}) in the dimension reduced form.

\noindent \textbf{Case 2: $\mathbf{|B_1|-|B_3|=1.}$} We will need to deal with two sub-cases: $|B_1|=2, |B_3|=1$ and $|B_1|=3, |B_3|=2$. We first start with the case $|B_1|=2, |B_3|=1$, in that case $B_3=\{y\}$. 

Suppose first that $B_2\neq \emptyset$. Repeated use of (\ref{l:Aexpand}) gives
\begin{align}
A[B_1,\emptyset,\{y\},y]=\sum_{B\subset B_2}  A[B_1\cup B,B_2\setminus B,\{y\}\cup B,y] \label{case2eq}
\end{align}
which means that $A[B_1,B_2,\{y\},y]-A[B_1, \emptyset, \{y\},y]$ can be written as the sum of building blocks with $|B_1|=3, |B_3|=2$ or $|B_1|>3$. 
Thus  that we may assume that $B_2=\emptyset$ and our operator is $A[\{0,y\},\emptyset, \{y\},y]$. 
Note that $\chi^{1/2}A[\{0,y\},\emptyset, \{y\},y]$ is exactly 
the  $A_+$ part for an exclusion with jumps of size $y$ (with rate 1).  
By Theorem 5.2 of \cite{Varadhan_hdl}, the  generator of any local mean zero exclusion process is sectorial.  Applying this with $f\in \cM_n, g\in \cM_{n+1}$ we get that the previous inequality also holds for just the $A_+$ part of the generator as well (and then it also holds for the scalar product $\ll\cdot, \cdot\rr$). Since $A[\{0,y\},\emptyset, \{y\},y]-y A_{+,TASEP}$ is the $A_+$ part of the generator of a mean zero exclusion (up to a constant multiplier) this implies that we may 
replace $A[\{0,y\},\emptyset, \{y\},y]$ with $y A_{+,TASEP}$. This shows that the contribution of the terms with  $|B_1|=2, |B_3|=1$ can be written as a sum of $c_+ A_+$ with some constant $c_+$, $A[B_1,B_2,B_3,y]$ with $|B_1|=3, |B_3|=2$ and a term which is sectorial. 

In order to determine the value of $c_+$ we just need to collect all the coefficients of the $|B_1|=2, |B_3|=1$ terms (weighted by $\chi^{-1/2} y$). We  do this  by consulting  the table (\ref{ratetable}). Note that $|B_1|=2, |B_3|=1$ can happen if $\ell=0, 1$ or $2$ and we get 
\begin{align}\label{cplus}
c_+=\sum_{y} y c_{y,\emptyset}  -\rrr \sum_{y,|\Lambda|=1}  y c_{y,\Lambda}  -\sum_{y,|\Lambda|=2} y c_{y,\Lambda}. 
\end{align}
To summarize: the contribution of all the terms from Case 2 can be written as $c_+ A_+$, $A[B_1,B_2,B_3,y]$ with $|B_1|=3, |B_3|=2$ and sectorial terms.

Now we turn to the $|B_1|=3, |B_3|=2$ terms. Note that on $\cM_2^{(1)}$ these will vanish after the dimension reduction by Lemma \ref{l:Asimple}.
To simplify the $|B_1|=3, |B_3|=2$ terms on $\cM_3^{(1)}$ we may assume that $y=1$ (by the same argument as in Case 1). Then $B_3=\{-a,1\}$ or $\{1,a\}$ with $a>0$. Using similar arguments as in Case 1 we may assume that $B_1=\{-a,0,1\}$ and $B_2=\{-a+1,\dots,-1\}$ in the first case and $B_1=\{0,1,a\}$, $B_2=\{2,\dots,a-2\}$ in the second case (the error being sectorial). This gives the following possibilities for $A[B_1,B_2,B_3,1]$ in the dimension reduced form (acting on $\mathcal M_3$):
\begin{align}
T_1 f(y_1,y_2,y_3)&=\ind(y_1=a-1,y_2=0)(f(a-1,y_3+1)-f(a,y_3))\\
\notag&\hskip80pt+ \ind(y_2=a-1,y_3=0)(f(y_1,a-1)-f(y_1,a))\\
T_2 f(y_1,y_2,y_3)&=\ind(y_1=0,y_2=a-1)(f(a-1,y_3)-f(a,y_3))\\
\notag&\hskip80pt+ \ind(y_2=0,y_3=a-1)(f(y_1-1,a-1)-f(y_1,a))
\end{align}
Since $|a|\le K$ by the definition  of $\cM_3^{(1)}$ we may further simplify these operators on $\cM_3^{(1)}$:
\begin{align}\label{T341}
T_1 f(y_1,y_2,y_3)&=\ind(y_1=a-1,y_2=0)(f(0,y_3+1)-f(0,y_3))\\
T_2 f(y_1,y_2,y_3)&= \ind(y_2=0,y_3=a-1)(f(y_1-1,0)-f(y_1,0))\label{T342}
\end{align}

Finally, we note that by (\ref{flux_der2})  we have that $\tfrac12{\chi j''(\rho)}$ is equal to  the sum of the second order coefficients in the orthonormal representation of the microscopic flux function $W$. By (\ref{flux}) we can write $W$ as  
\begin{align}
W=\sum_y \sum_{z\in Z_y} \textup{sgn}(y) r(y,\tau_{-z}\eta)(\chi+\chi^{1/2} \heta_{z+y}(1-\rho)-\rho \chi^{1/2} \heta_z-\chi \heta_{\{z,z+y\}})
\end{align}
 where $|Z_y|=|y|$. 
The sum of the coefficients of the second order terms come from the constant, first order and second order coefficients in the representation of the rate functions $r(y,\eta)$ and it gives $
-\chi \sum_{y} y c_{y,\emptyset}+\chi^{1/2}(1-2\rho) \sum_{y,|\Lambda|=1} y c_{y,\Lambda}+\chi \sum_{y, |\Lambda|=2} y c_{y,\Lambda}$.
Recalling the definition of $\rrr$ from (\ref{kappa}) we get that this is exactly equal to $-\chi c_+$ which shows that $c_+=-{j''(\rho)}/{2}$.

\noindent \textbf{Case 3: $\mathbf{|B_1|-|B_3|=0.}$} We will need to consider two cases: $|B_1|=1$ or  2. We first deal with the $|B_1|=1$ terms. Note that in that case $|B_2|\ge 1$.  Using the same arguments as in the previous case we can show that $A[B_1,B_2,\{y\},y]$ is equal to the  sum of $y A[\{0\},\{1\},\{1\},1]$, some terms with $|B_1|=|B_3|=2$  and sectorial terms. Note that $A[\{0\},\{1\},\{1\},1]$ is sectorial by Lemma \ref{l:A0bnd} so we only have to describe the terms with $|B_1|=|B_3|=2$. We only need to do this on $\cM_3^{(1)}$ as the dimension reduced forms of these terms vanish on $\cM_2^{(1)}$ by Lemma \ref{l:Asimple}.

If $|B_1|=|B_2|=2$ then by the previous arguments we may rewrite $A[B_1,B_2,B_3,y]$ as the linear combination of terms $A[\tilde B_1, \tilde B_2, \tilde B_3, 1]$ with:
\begin{align}
\tilde B_1&=\tilde B_3=\{1,a\}, \qquad \tilde B_2=[0,a]\setminus \tilde B_1, \qquad a>0, \qquad \textup{or}\\
\tilde B_1&=\tilde B_3=\{-a,1\}, \qquad \tilde B_2=[-a,1]\setminus \tilde B_1, \qquad a>0.
\end{align}
and some sectorial terms.
The dimension reduced form of the operators acting on  $\mathcal M_3$ are easy to write down:
\begin{align}
T_1 f(y_1,y_2)&=\ind(y_1=a-1)(f(a,y_2)-f(a-1,y_2))+\ind(y_1>0,y_2=a-1)(f(y_1-1,a)-f(y_1,a-1))\\
T_2 f(y_1,y_2)&=\ind(y_1=a)(f(a-1,y_2+1)-f(a,y_2))+\ind(y_2=a)(f(y_1,a-1)-f(y_1,a))
\end{align}
Moreover, these can be further simplified if we restrict ourselves to $\cM_3^{(1)}$:
\begin{align}\label{T331}
T_1 f(y_1,y_2)&=\ind(y_1>0,y_2=a-1)(f(y_1-1,0)-f(y_1,0))\\
 T_2 f(y_1,y_2)&=\ind(y_1=a)(f(0,y_2+1)-f(0,y_2))\label{T332}
\end{align}

\noindent \textbf{Case 4: $\mathbf{|B_1|-|B_3|=-1.}$} 
The contribution of these terms give the part of the operator $A$ which lowers the degree by one: $\mathcal M_n\to \mathcal M_{n-1}$. Since $A^*=-A$ from this it follows that this will be minus the adjoint of the part which raises the degree by one (which was discussed in Case 2). From this it is not hard to check that we can rewrite the contribution of these terms as $c_+ A_-$ (with  $c_+$ from (\ref{cplus}), and $A_-$ being the $A_-$ part of the TASEP), building blocks with $|B_1|=2, |B_3|=3$ and sectorial terms. However, by Lemma \ref{l:Asimple} the terms with  $|B_1|=2, |B_3|=3$ will vanish after the dimension reduction so we are only left with $c_+ A_-$.

\noindent \textbf{Case 5: $\mathbf{|B_1|-|B_3|=-2.}$}  
Since $|B_1|\le 3$ and $|B_2|\ge 2$ the contribution of these terms is sectorial.

Collecting all the cases we can see that Lemma \ref{l:Abnd1d} is proved, we just have to show that the contribution of the $|B_1|=3, |B_2|=2$ and $|B_1|=|B_3|=2$ terms (see (\ref{T341}), (\ref{T342}), (\ref{T331}), and (\ref{T332})) is a linear combination of the operators (\ref{A34}) and (\ref{A33}) plus some sectorial terms. 
\smallskip

We first look at  the contribution of the $|B_1|=B_3|=2$ terms from  (\ref{T331}) and (\ref{T332}). We will start by proving that if we define $\tilde T_1$ and $\tilde T_2$ as (\ref{T331}) and (\ref{T332}) with $a=1$ then $T_1-\tilde T_1$ and $ T_2-\tilde T_2$ are sectorial. From the variational formula
\begin{align}\notag
&\ll (T_1-\tilde T_1) f, (-\SSEP)^{-1} (T_1-\tilde T_1)f\rr=\sup_{g\in \cM_3} 2\ll g, (T_1-\tilde T_1) f \rr-\ll g, (-S) g\rr\\
&\hskip60pt=\sup_{g\in \cM_3} 2 \sum_{y_1>0} (g(y_1,a-1)-g(y_1,0))(f(y_1-1,0)-f(y_1,0))-\ll g, (-\SSEP) g\rr\\
&\hskip60pt\le C\sum_{y_1>0}(f(y_1-1,0)-f(y_1,0))^2\le 2C \ll f, (-\SSEP) f\rr.
\end{align}
The first inequality follows from $\ll g, (-\SSEP) g\rr\ge C \sum_{y_1>0}(g(y_1,a-1)-g(y_1,0))^2$ which is a consequence of the Cauchy-Schwarz inequality (note that $C$ depends only on $|a|\le K$).
The same argument works for $ T_2-\tilde T_2$. We will further modify $\tilde T_1$ and show that it can be well approximated by 
\begin{align}
\tilde T_1'f(y_1,y_2)=\ind(y_2=0)(f(y_1,0)-f(y_1+1,0))
\end{align}
(i.e.~the difference is sectorial). This follows a similar argument:
\begin{align}\notag
&\ll (\tilde T_1-\tilde T_1') f, (-\SSEP)^{-1} (\tilde T_1-\tilde T_1')f\rr=\sup_{g\in \cM_3} 2\ll g, (\tilde T_1-\tilde T_1') f \rr-\ll g, (-\SSEP) g\rr\\
&\hskip60pt=\sup_{g\in \cM_3} 2 \sum_{y_1>0} g(y_1,0)(f(y_1-1,0)-2f(y_1,0)+f(y_1+1,0))-\ll g, (-\SSEP) g\rr\\
&\hskip60pt=\sup_{g\in \cM_3} 2 \sum_{y_1>0} (g(y_1,0)-g(y_1+1,0))(f(y_1+1,0)-f(y_1,0))-\ll g, (-\SSEP) g\rr\\
&\hskip60pt\le C\sum_{y_1>0}(f(y_1+1,0)-f(y_1,0))^2\le 2C \ll f, (-\SSEP) f\rr.
\end{align}
We used $f(0,0)=f(1,0)=f(2,0)$ which follows from $f\in \cM_3^{(1)}$. Now note that $\tilde T_1'-\tilde T_2$ is exactly the operator (\ref{A33}) without the constant $C$. Thus we just have to prove that the sum of the coefficients of the operators $T_1$ from (\ref{T331}) is minus the sum of the coefficients of the operators $T_2$ from (\ref{T332}).

This requires a bit of a book-keeping, using the arguments preceding Lemma \ref{l:bookkeeping}  one can show that we only get $|B_1|=|B_3|=2$ terms when $|\Lambda|=1$, 2 or 3 and we have the following contributions:

\noindent \textbf{Case 1.} {$\Lambda=\{a\}$, $a\neq y, 0$: The contribution is} $\sqrt{\chi}(\kappa^2-2)c_{y,\Lambda}A[\{q_1,q_1+1\},\{q_1+2,\dots,q_2 \},\{y,a\},y]$ where $q_1=\min(0,y,a)$, $q_2=\max(0,y,a)$

\noindent  \noindent \textbf{Case 2.}  $\Lambda=\{a_1,a_2\}$, $a_i\neq y, 0$: The contribution is\\ 
$\sum_{i=1}^2 2\sqrt{\chi}\kappa c_{y,\Lambda}A[\{q_{1,i},q_{1,i}+1\},\{q_{1,i}+2,\dots,q_{2,i} \},\{y,a_i\},y]$ with $q_{1,i}=\min(0,y,a_i)$, $q_{2,i}=\max(0,y,a_i)$,

\noindent  \noindent \textbf{Case 3.}  $\Lambda=\{a_1,a_2,a_3\}$, $a_i\neq y, 0$: The contribution is\\
$\sum_{i=1}^3 \sqrt{\chi} c_{y,\Lambda}A[\{q_{1,i},q_{1,i}+1\},\{q_{1,i}+2,\dots,q_{2,i} \},\{y,a_i\},y]$ with $q_{1,i}=\min(0,y,a_i)$, $q_{2,i}=\max(0,y,a_i)$

One can also check by the preceding computations that $A[B_1,B_2,\{y,a\},y]$ is equal to $\alpha \tilde T_1'-\beta \tilde T_2$ plus sectorial terms and $\alpha-\beta=|a|-|a-y|$. We will show that in each of the three cases listed above the sum of the $\alpha$ coefficients is same as the sum of the $\beta$ coefficients this will prove $\alpha \tilde T_1'-\beta \tilde T_2$ is a constant multiple of the operator (\ref{A33}).
 In order to do this we just have to prove that for any fixed $k$ we have
\begin{align}
\sum_{y, |\Lambda|=k} c_{y,\Lambda}\sum_{a\in \Lambda}(|a|-|a-y|)=0.
\end{align}
But miraculously, this follows easily from the divergence condition (\ref{grad_cond}), similarly to the last argument in the proof of Lemma \ref{l:miracle}). This completes the first part of the proof.

What  left is to show that the contribution of the $|B_1|=3, |B_2|=2$ terms (see (\ref{T341}), (\ref{T342})) is a constant multiple  of the operator (\ref{A34}) plus some sectorial terms. This can be done similarly as the $|B_1|= |B_2|=2$ case, we will leave the details to the reader. 
\end{proof}

\begin{proof}[Proof of Lemma \ref{l:A2d}]  The proof is very similar to the proof of Lemma \ref{l:Abnd1d}. We start with the representation
$
A=\sum c_{B_1,B_2,B_3,y} A[B_1,B_2,B_3,y]$ and note that since we are in two dimensions we only need to consider the terms with $|B_1|\le 2$, the others being sectorial by Lemma \ref{l:Abnd3}. We will only need to worry about the parts of the generator which leave the degree the same or raise it, the other parts can be computed from the skew adjoint property $A^*=-A$. This means that we only have to consider the following three cases: $|B_1|=2, |B_3|=1$, $|B_1|=|B_3|=1$ and $|B_1|=|B_3|=2$. 

We start with the contribution of the terms $|B_1|=2, |B_3|=1$. Using the same argument as in Case 2 of the proof of Lemma \ref{l:Abnd1d} we may assume that $B_1=\{0,y\}$, $B_2=\emptyset$ and $B_3=y$ (the difference being sectorial). Then $\chi^{1/2} A[B_1,B_2,B_3,y]$ is the $A_+$ part of an exclusion with jump $y$ (and rate 1). Using the fact the $A_+$ of a mean zero exclusion process is sectorial (see the arguments following  (\ref{case2eq}))  
 we may replace this with the $A_+$ from an exclusion with \emph{nearest neighbor} jump rates with drift $y$ (again, the difference is sectorial).  The total contribution of these operators will also going to be the $A_+$ of a nearest neighbor exclusion process and the drift is given by (\ref{cplus}) (where $y$  now runs through $\Z^2$ vectors).

The $|B_1|=|B_3|=1$ terms can be treated as the similar terms in the one dimensional case: $A[B_1, B_2, B_3,y]$ can be replaced with linear combinations of $A[\{0\},\{e_i\},\{e_i\},e_i]$ ($i=1,2$), terms with $|B_1|=|B_3|=2$ and sectorial terms. The operator $A[\{0\},\{e_i\},\{e_i\},e_i]$ is sectorial by Lemma \ref{l:A0bnd}.
 Finally, the building blocks with $|B_1|=|B_3|=2$ are sectorial on $\cM_2^{(2)}$ by Lemma \ref{l:Asimple}. 

Thus $A$ can be written as a linear combination of $A_++A_-$ coming from a nearest neighbor asymmetric TASEP  and sectorial terms which proves the lemma.
\end{proof}

\section{Proof of the hard core removal lemma}\label{s:hardcore}

The proofs for the one dimensional case will depend on the following lemma.
\begin{lemma}[Foldout lemma]
\label{l:foldout}
For any $G_1: \Z_+^2\to \R,  G_2:\Z_+^4\to \R$ satisfying
\begin{align}
G_1(x_1,x_2)&=g_1(x_1-x_2)\ind(x_1=0 \textup{ or } x_2=0),\\G_2(x_1,x_2,x_3,x_4)&= g_2(x_1-x_2,x_3)\ind(x_1=x_4=0\textup{ or } x_2=x_4=0)
\end{align}
we have the bounds
\begin{align}
\ll G_1, (\lambda-\SSEP)^{-1}  G_1\rr\le C  \ll \tilde g_1 , (\lambda-\Delta)^{-1} \tilde g_1 \rr, \quad \ll G_2, (\lambda-\SSEP)^{-1}  G_2\rr\le C  \ll \tilde g_2 , (\lambda-\Delta)^{-1} \tilde g_2\rr\label{foldingbnd}
\end{align}
where $\tilde g_1:\Z^2\to \R$, $\tilde g_2:\Z^4\to \R$ are defined as 
\begin{align}\label{eq:tg1}
\tilde g_1(x_1,x_2)&=g_1(x_1)\ind(x_2=0)\\
\tilde g_2(x_1,x_2,x_3,x_4)&=g_2(x_1,|x_2|-\ind(x_2<0))\ind(x_3=x_4=0).\label{eq:tg2}
\end{align}
and $\Delta$ denotes the usual lattice Laplacian in the appropriate dimension. 
\end{lemma}
\begin{proof}
Consider the map $\mathcal T: (x,y)\to (x-y,\min(x,y))$ which gives a one-to-one correspondence between $\Z_+^2$ and $\Z\times \Z_+$. This is like folding out $\Z_+^2$ into $\Z\times \Z_+$.
It is Lipschitz, in particular, nearest neighbor points are mapped to pairs with $l_1$ distance of 1 or 2.
Let $q:\Z_+^2\to \R$ be an arbitrary (compactly supported) function. Define $\tilde q:\Z^2\to \R$ as
\begin{align}
\tilde q(x,y)=\begin{cases}
q(\mathcal T^{-1}(x,y))\quad&\textup{ if } y\ge 0,\\[3pt]
q(x, -1-y)\quad&\textup{ if } y<0.
\end{cases}
\end{align}
Then it is easy to check that 
$\ll G_1, q\rr=\frac12 \ll \tilde g_1 ,\tilde q\rr$, 
$\ll q,(\lambda-\SSEP)q\rr\ge C \ll \tilde q, (\lambda-\Delta) \tilde q\rr$,
and the first part of (\ref{foldingbnd}) follows from the variational formula.
The second part  of (\ref{foldingbnd}) is similar. We start by the 4-dimensional version of the `folding-out' trick. Let $q:\Z_+^4\to \R$ be an arbitrary (local) function. Define $\tilde q: \Z^4\to \R$ by
\begin{align}
\tilde q(x_1,x_2,x_3,x_4)=\begin{cases}
&q(T^{-1}(x_1,x_2),x_3,x_4)\qquad\textup{ if } x_2\ge0, x_3\ge0, x_4\ge 0\\[3pt]
&q(x_1,|x_2|-\ind(x_2<0),|x_3|-\ind(x_3<0),|x_4|-\ind(x_4<0))\,\,\textup{ otherwise}
\end{cases}
\end{align}
Then
$
\ll G_2, q\rr=
\frac12 \ll \tilde  g_2,  \tilde q \rr$, $\ll q,(\lambda-S_4)q\rr\ge C \ll \tilde q, (\lambda-\Delta) \tilde q\rr$
and the proof again follows from the variational formula.
\end{proof}


\begin{proof}[Proof of Lemma \ref{1dhardcore}]
In order to prove (\ref{hardcore11}) note that from the definition of $F$ we get 
\begin{align}
\l F, (\lambda-\Delta) F\r=4 \lambda \sum_{x \in \Z_+^2} f(x)^2+2 \sum_{x\in \Z_+^2, i=1,2} (f(x)-f(x+e_i))^2
\end{align}
where   $e_1, e_2$ are the usual unit vectors. For $\l f, (\lambda-\SSEP) f\r$ we get
\begin{align}
 \lambda \sum_{x \in \Z_+^2} f(x)^2+\frac12  \sum_{x\in \Z_+^2, i=1,2} (f(x)-f(x+e_i))^2+\frac12  \sum_{x, x+e_1-e_2\in \Z_+^2} (f(x)-f(x+e_1-e_2))^2
\end{align}
The last term can be bounded by the second one using Cauchy-Schwarz which shows that $\l f, (\lambda-\SSEP) f\r\le C \l F, (\lambda-\Delta) F\r$.
Now $ \l F, (\lambda-\Delta) F\r$ can be bounded by the right hand side of  (\ref{hardcore11}) by simple Fourier computations.

To prove (\ref{hardcore12}) we consider the following modification of $A_{TASEP}$:
\begin{align}
\bar A f(y_1,y_2)=\ind(y_1=0, y_2>0)(f(y_2)-f(y_2-1))+\ind(y_2=0)(f(y_1)-f(y_1+1)).
\end{align}
The difference $T=\bar A-A_{TASEP}$ is given by 
\begin{align}
Tf(y_1,y_2)=\ind(y_1=0, y_2>0)(2 f(y_2)-f(y_2-1)-f(y_2+1))
\end{align}
assuming $f(1)-f(0)=1$, which we may. Using the variational formula
\begin{equation}\label{185}
\l Tf, (\lambda-\SSEP)^{-1} Tf\r\le \sup_g \sum_y 2(2f(y+1)-f(y)-f(y+2))g(x,y)-\sum_y (g(y)-g(y+1))^2
\end{equation}
Summing by parts, it is easy to see that this is bounded by $ \l f, (-\SSEP) f\r$.
To bound $\l \bar A f, (\lambda-\SSEP)^{-1} \bar A f\r$ we use Lemma \ref{l:foldout} with 
 the function 
$g_1(x)=
f(x-1)-f(x)$ if $x>0$ and $g_1(x)=
f(-x+1)-f(-x)$ if $x\le 0$, which gives
\begin{equation}
\l \bar A f, (\lambda-\SSEP)^{-1} \bar A f\r\le C' \sup_{\tilde g} 2\l  (F(x)-F(x+1))\ind(y=0), \tilde g\r -\l \tilde g, (\lambda-\Delta) \tilde g\r.
\end{equation}
The right hand side is easily computed using Fourier transform as
\begin{align}
C' \int_{-\pi}^{\pi}  \int_{-\pi}^{\pi}  
\frac{|\hat F(t_1)|^2 |e^{i t_1}-1|^2}{\lambda+|e^{i t_1}-1|^2+|e^{i t_2}-1|^2} dt_1dt_2
\le C'' \int_{-\pi}^{\pi} |\hat F(t)|^2 \frac{t^2}{\sqrt{\lambda}}dt,
\end{align}
which proves (\ref{hardcore12}).

%

Finally we prove (\ref{hardcore13}) where $\tilde A$ is defined in (\ref{A24}). We will assume $C=1$.
It is easy to see that
\begin{align}\label{183}
\tilde A f(y_1,y_2,y_3)=\ind(y_2=0) A_+f(y_1,y_3)+\ind(y_2=y_3=0)(2f(y_1+1)-f(y_1)-f(y_1+2)).
\end{align}
By the argument used in (\ref{185}), the $H_{-1}$ norm of the second part is bounded by $\l f, (-\SSEP) f\r$. 
From the variational formula,
\begin{align}
\l \ind(y_2=0) A_+f(y_1,y_3),(\lambda-\SSEP)^{-1} \ind(y_2=0) A_+f(y_1,y_3)\r\le \l  A_+f,(\lambda-\SSEP)^{-1}  A_+f\r
\end{align}
Using (\ref{hardcore11}) and (\ref{hardcore12})  the bound (\ref{hardcore13}) now follows. 
\end{proof}

\begin{proof}[Proof of Lemma \ref{1dhardcore_infl}]
The proof is similar to that of Lemma \ref{1dhardcore}, but we need to handle a few more terms. We  have $\l f, (\lambda-\SSEP) f\r=4 \l F, (\lambda-\Delta) F\r$ and $ \l F, (\lambda-\Delta) F\r$ can be bounded by the right hand side of  (\ref{hc21}) by Fourier computations. In order to get (\ref{hc22}) we will need to bound $\l \tilde A_i f , (\lambda-\SSEP)^{-1} \tilde A_i f \r$ for $i=1, 2, 3$ where $\tilde A_i$ are defined in (\ref{A33}), (\ref{A34}) and (\ref{A35}). We will start with $i=3$. 

We first separate $\tilde A_3 f$ into two parts, we will bound the $H_{-1}$ norms separately. 
\begin{align}
\tilde A_{3,1}f(y_1,y_2,y_3,y_4)&=\ind(y_1=y_2=0)(f(y_3+1,y_4)-f(y_3,y_4))\\\notag&\qquad+\ind(y_2=y_3=0)(f(y_1+1,y_4)-f(y_1+2,y_4)),\\
\tilde A_{3,2}f(y_1,y_2,y_3,y_4)&=\ind(y_2=y_3=0)(f(y_1+1,y_4+1)-f(y_1+1,y_4))\\&\qquad+\ind(y_3=y_4=0)(f(y_1,y_2+1)-f(y_1,y_2+2))\notag
\end{align}
We modify these a little but by introducing 
\begin{align}
\bar A_{3,1}f(y_1,y_2,y_3,y_4)&=\ind(y_1=y_2=0)(f(y_3+1,y_4)-f(y_3,y_4))\\\notag&\qquad+\ind(y_2=y_3=0,y_1>1)(f(y_1-2,y_4)-f(y_1-1,y_4))\\
\bar A_{3,2}f(y_1,y_2,y_3,y_4)&=\ind(y_2=y_3=0)(f(y_1,y_4+1)-f(y_1,y_4))\\\notag&\qquad+\ind(y_3=y_4=0,y_2>1)(f(y_1,y_2-2)-f(y_1,y_2-1))
\end{align}
The fact, that $\l (\bar A_{3,i} -\tilde A_{3,i})f, (\lambda-\SSEP)^{-1} (\bar A_{3,i} -\tilde A_{3,i})f \r$ for $i=1,2$ are bounded by $C \ll f, (-\SSEP) f \rr$
follows the same way as  by the argument used in (\ref{185}).

To bound the $H_{-1}$ norm of $\tilde A_{3,1}f$ we use Lemma (\ref{l:foldout}) with the following setup: 
\begin{align}\nonumber
G_2(x_1,x_2,x_3,x_4)&=\tilde A_{2,1}f(x_2,x_4,x_1,x_3)\\
g_2(x,y)&=\begin{cases}
f(x+1,y)-f(x,y)=F(x+1,y)-F(x,y)&\qquad \textup{ if } x\ge  0
\\
f(-x-2,y)-f(-x-1,y)=F(x+1,y)-F(x,y) &\qquad \textup{ if } x<0
\end{cases}
\end{align}
where we used $f\in \cM_3^{(1)}$ for the $x=0$ case. 
Then $\tilde g_2$ defined in (\ref{eq:tg2}) is given by
\begin{equation}
\tilde g_2(x_1,x_2,x_3,x_4)=(F(x_1+1,x_2)-F(x_1,x_2))\indd{x_3=x_4=0}.
\end{equation} 
Now $\l \tilde g_2, (\lambda-\Delta)^{-1} \tilde g_2 \r$ can be computed explicitly using Fourier-transform of $\hat F$ which leads to 
\begin{align}
\l \tilde g_2, (\lambda-\Delta)^{-1} \tilde g_2 \r&\le C\int_{[-\pi,\pi]^4} \frac{t_1^2}{\lambda+t_1^2+t_2^2+t_3^2+t_4^2} |\hat F(t_1,t_2)|^2 dt_1 dt_2 dt_3 dt_4.
\end{align}
Integrating out $t_3$ and $t_4$ gives the first term on the right hand side of (\ref{hc22}).
We can bound the  $H_{-1}$ norm of $\tilde A_{3,2}f$ exactly the same way (using $f(x,y)=f(y,x)$). 

To bound $\l \tilde A_1 f , (\lambda-\SSEP)^{-1} \tilde A_1 f\r $ note that with $\tilde f(y)=f(0,y)$ we have  $\tilde A_1 f(y_1,y_2)=A_{TASEP} \tilde f(y_1,y_2)$ and 
we could use the bounds from the proof of Lemma \ref{1dhardcore}. The role of $\hat F(t)$ is played by $\int_{-\pi}^\pi \hat F(t_1,t_2) dt_1$ which gives the second term on the right hand side of (\ref{hc22}). 

Finally, one can show that $\l \tilde A_2 f , (\lambda-\SSEP)^{-1} \tilde A_2 f \r$ is bounded by $C \l \tilde A_1 f , (\lambda-\SSEP)^{-1} \tilde A_1 f\r $ using the same argument as in the end of the proof of Lemma \ref{1dhardcore}. This finishes the proof of Lemma \ref{1dhardcore_infl}.
\end{proof}

\begin{proof}[Proof of Lemma \ref{2dhardcore}] 
Consider the symmetric extension $ f_* $ of $f$ on $(\Z^2)^2\to \R$. This is given by 
$
f_*(x_1,y_1,x_2,y_2)=f(\{(x_1,y_1),(x_2,y_2)\}) 
$ if $(x_1,y_1)\neq (x_2,y_2)$ and zero otherwise. By Lemma 4.3 of \cite{LQSY} we can bound $\l f, (\lambda -\SSEP) f \r $ with $C \l  f_*, (\lambda-\Delta) f_* \r$ and the same estimate holds for the quadratic forms with $\ll\cdot, \cdot \rr$. A simple Fourier computation shows that
\begin{equation}
\ll f_*, (\lambda -\Delta) f_*\rr \le C\int_{[-\pi,\pi]^2} (\lambda+|u|^2) \hat f_*(u,-u) du.
\end{equation}
Here $\hat f_*(u,v)$ is the Fourier transform of $f_*$, and it is clear that $\hat f_*(u,-u)=\hat F(u)$ where $F$ is defined in (\ref{deF}). This proves (\ref{2dhardc1}). 

The second part of the lemma follows directly from Lemma 3.2 and 4.5 of \cite{LQSY}. Note that although in the paper they only deal with the $(a,b)=(1,0)$ drift, the arguments can be extended to the general case without any difficulty. 
\end{proof}

\section{Appendix:  Green-Kubo formula}

\begin{lemma}\label{l:GK}  Let  $
C_{ii}=\sum_y y_i^2 \, E \, 
r(y,  \eta)$.  Then
\begin{equation}\label{GK}
 t D_{ii}(t)=C_{i,i}t   +2 \chi  \int_0^t \int_0^s \la\la
w_i (s) ,e^{u L^*
} w_i(s) \ra\ra du \,ds-\int_0^t \int_0^s \la\la
v_i(s) ,e^{u L^*} v_i(s) \ra\ra du \,ds,
\end{equation} \end{lemma}

{\it Proof of Lemma \ref{l:GK}}.
We will only show the $d=1$ case, the general case is similar.  The first part of the proof is standard.
We start with the martingales \begin{equation}
M_x(t)= \eta_x(t)-\eta_x(0)-\int_0^t \sum_{y}  \left( 
r(y, \tau_{x-y}\eta) \eta_{x-y}(1-\eta_x)
-
r(y, \tau_x \eta) \eta_x(1-\eta_{x+y})\right) ds
\end{equation}
 Using the notation $\nabla_z f(x)=f(x+z)-f(x)$ and the identity  $\sum_x (\nabla_z a) b=\sum_x (\nabla_{-z} b) a$ we get 
 \begin{eqnarray*}
\sum_x x^2 \la \eta_x(s), \eta_0(0)-\rho \ra&=&\sum_x x^2\la\eta_x(0)+M_x(t),\eta_0(0)-\rho \ra\\
&&+\sum_x x^2 \la
\int_0^t \sum_{y}  \left( 
\nabla_{-y} \left(
r(y, \tau_x \eta) \eta_x(1-\eta_{x+y})\right)\right)ds,\eta_0(0)-\rho\ra.\end{eqnarray*}
The first term vanishes by the independence of $\{\eta_x(0)\}$ under $\pi_\rho$ and the martingale property.  
Now we use the generator $L^*$ of the time reversed process to get
\begin{equation}
\eta_x(0)=\eta_x(s)+\int_0^s \sum_{|z|\le K}  \left( 
r(z, \tau_{x}\eta^*) \eta^*_{x+z}(1-\eta^*_x)
-
r(z, \tau_{x-z} \eta^*) \eta^*_x(1-\eta^*_{x-z})\right) ds+M_x^*(s)
\end{equation}
where $\eta^*(u)$ is the time reversed lattice gas process started from $\eta^*(0)=\eta(s)$. This gives 
\begin{align}
&\sum_x x^2 \la \eta_x(s), \eta_0(0)-\rho \r=
\sum_{x,y} \int_0^t\la (-2xy+y^2) \left(
r(y,  \eta) \eta_0(1-\eta_{y})\right),\eta_{x}-\rho+M^*_{x}(s)\ra 
ds \label{xxx1st}
\\
&~+\sum_{x,y} \int_0^t\int_0^s\left \la (-2xy+y^2) \left(
r(y,  \eta) \eta_0(1-\eta_{y})\right),  - \sum_{|z|\le K}  \nabla_{-z}\left( 
r(z, \tau_{x}\eta^*) \eta^*_{x+z}(1-\eta^*_{x})\right)\right\ra du ds\label{xxx2nd}
\end{align}
In the last term $\eta_x=\eta_x(s)$ and $\eta^*_x=\eta^*_x(u)$ and $\nabla_{-z}$ acts in the $x$ variable.  We will compute (\ref{xxx1st}) and (\ref{xxx2nd}) separately.

We start with (\ref{xxx2nd}). By moving the gradient $\nabla_{-z}$ onto the first term we get
\begin{align}
&\hskip-20mm\sum_{x,y,z} 2yz \int_0^t \int_0^s \left\la r(y,  \eta) \eta_0(1-\eta_{y});r(z, \tau_{x}\eta^*) \eta^*_{x+z}(1-\eta^*_{x})\right\ra du\, ds\nonumber\\
&=2 \int_0^t \int_0^s \ll  \sum_y yr(y,  \eta) \eta_0(1-\eta_{y}),\sum_z zr(z, \eta^*) \eta^*_{z}(1-\eta^*_{0}\rr du\, ds
\label{secondsum}
\end{align}
%
This gives
\begin{equation}
{\chi}^{-1}\sum_x x^2 \la \eta_x(s), \eta_0(0)-\rho \ra=C t+2 \chi^{-1} \int_0^t \int_0^s \la\la
W(s) ,e^{u L^*} W^*(s) \ra\ra du \,ds
\end{equation}
Using $\l a,b \r=\l (a+b)/2,(a+b)/2\r-\l (a-b)/2, (a-b)/2 \r$ and introducing $v:=W- W^*, \tilde W:=W+W^*$ we can rewrite the integral term as
\begin{eqnarray*}
\int_0^t \int_0^s \la\la
W(s) ,e^{u L^*} W^*(s) \ra\ra du \,ds&=&\int_0^t \int_0^s \la\la
\tilde W(s) ,e^{u L^*} \tilde W(s) \ra\ra du \,ds -\int_0^t \int_0^s \la\la
v (s) ,e^{u L^*} v(s) \ra\ra du \,ds
\end{eqnarray*}
Recall that 
$
w=\tilde W-\mathcal P_1 \tilde W-\la \tilde W \ra$.
By the conservation of $\sum \eta$ we have 
$
\ll \mathcal P_1 \tilde W, e^{u L^*} \mathcal P_1 \tilde W(s)\rr
=\ll \mathcal P_1 \tilde W, \mathcal P_1 \tilde W(s)\rr$ and $\ll w, e^{u L^*} \mathcal P_1 \tilde W(s)\rr=0$.
This gives
\begin{equation}
2  \int_0^t \int_0^s \la\la
\tilde W(s) ,e^{u L^*} \tilde W(s) \ra\ra du \,ds=2 \int_0^t \int_0^s \la\la
 w(s) ,e^{u L^*}  w(s) \ra\ra du \,ds+ t^2 \ll \mathcal P_1 \tilde W, \mathcal P_1 \tilde W\rr
\end{equation}
In order to compute the last term we introduce $\eta_\Lambda=\prod_{x\in \Lambda} \eta_x$. If $|\Lambda|=k$ then
\begin{equation}
\eta_\Lambda=\prod_{x\in \Lambda}(\chi^{1/2} \heta_x+\rho)=\rho^k+\chi^{1/2} \rho^{k-1} \sum_{x\in \Lambda} \heta_x+\textup{higher order terms}.
\end{equation}
Since $\tilde W$ is a polynomial we can write $\tilde W=\sum a_\Lambda \eta_\Lambda$. Then $\mathcal P_1 \tilde W=\chi^{1/2}  \sum_\Lambda c_\Lambda \rho^{|\Lambda|-1}\sum_{x\in \Lambda} \eta_x$ and
\begin{equation}
 \chi^{-1} t^2\ll \mathcal P_1 \tilde W, \mathcal P_1 \tilde W\rr=t^2 \ll \sum_\Lambda \rho^{|\Lambda|-1}\sum_{x\in \Lambda} \eta_x, \sum_\Lambda a_\Lambda \rho^{|\Lambda|-1}\sum_{x\in \Lambda} \eta_x\rr=t^2 \left(\sum_\Lambda a_\Lambda \rho^{|\Lambda|-1} |\Lambda|\right)^2.
\end{equation}
Note that $j(\rho)=E W=E \tilde W=\sum_\Lambda  a_\Lambda \rho^{|\Lambda|}$ which means that 
\begin{equation}
 \chi^{-1} t^2\ll \mathcal P_1 w_+, \mathcal P_1 w_+\rr=t^2 j'(\rho)^2.
\end{equation}


The computation of (\ref{xxx1st}) relies on a bit of a miracle, which we state as a lemma.
\begin{lemma}\label{l:miracle}
\begin{align}
{\chi}^{-1}\sum_{x,y} \int_0^t\la (-2xy+y^2) \left(
r(y,  \eta) \eta_0(1-\eta_{y})\right),\eta_{x}-\rho+M^*_{x}(s)\ra=
t \sum_y y^2 \la 
r(y,  \eta) \ra.
\end{align}
\end{lemma}

\begin{proof}
By stationarity, and because $M^*_x$ is a martingale the  sum on the right is equal to
\begin{eqnarray*}
t {\chi}^{-1}\sum_y \sum_{x} \la (-2xy+y^2) 
r(y,  \eta) \eta_0(1-\eta_{y})),\eta_{x}-\rho\ra.
\end{eqnarray*}
If we consider the $x=0, x=y$ terms then we get
\begin{equation}
t {\chi}^{-1}\sum_y y^2 \la 
r(y,  \eta) \eta_0(1-\eta_{y})(\eta_{0}-\eta_y)\ra=
t \sum_y y^2 \la 
r(y,  \eta) \ra
\end{equation}
For the other terms, by the independence of $\eta_i, \eta_j$ when  $i\neq j$  we have
\begin{equation}
t {\chi}^{-1/2}\sum_y \sum_{x\neq 0, y}(-2xy+y^2) \la 
r(y,  \eta) \eta_0(1-\eta_{y})\heta_{x}\ra=
t \chi^{1/2} \sum_y  \la 
r(y,  \eta), \sum_x (-2xy+y^2)\heta_x\ra\label{otherterm}.
\end{equation}
We just need to show that the term on the right is zero.
Let $\mathcal P_1$ be the projection to $\mathcal M_1$ and denote $\mathcal P_1 r(y,\eta)=\sum_x c_{y,x} \heta_x$. (We can assume that $c_{y,y}=c_{y,0}=0$.) Then 
\begin{eqnarray}\nonumber
\sum_y  \la 
r(y,  \eta), \sum_x (-2xy+y^2)\heta_x\ra&=&\sum_y \sum_x c_{y,x} (y^2-2xy)=\sum_{y,x} c_{y,x} ((y-x)^2-x^2)\\
&=&\sum_a a^2 \sum_{|y-x|=a} c_{y,x}-\sum_{a} a^2 \sum_{y, |x|=a} c_{y,x}\label{otherterm2}
\end{eqnarray}
By the divergence condition (\ref{grad_cond}) the sum $\sum_y \mathcal P_1 r(y,\eta)(\eta_y-\eta_0)$ must also be a gradient. By inspecting the coefficients of $\heta_{y,y+a}$ we get that this  is equivalent to 
\begin{equation}
\sum_{|y-x|=a} c_{y,x}=\sum_{y, |x|=a} c_{y,x} \qquad \textup{for all $a$}.
\end{equation}
This shows that (\ref{otherterm2}) and thus (\ref{otherterm}) will vanish.
\end{proof}

\noindent \textbf{Acknowledgements.} We thank Herbert Spohn for introducing us to speed change models. J.~Quastel was partially supported by the Natural Sciences and Engineering Research Council of Canada. B.~Valk\'o was partially supported by the NSF Grant DMS-09-05820 and the NSF CAREER award DMS-1053280.

\end{document}